\documentclass{article}
\usepackage{amssymb,amsmath,bm,geometry,graphics,graphicx,url,color}

\def\no{\noindent}
\def\pmatrix{\left(\begin{array}}
\def\endpmatrix{\end{array}\right)}

\newtheorem{theo}{Theorem}[section]
\newtheorem{lem}[theo]{Lemma}

\newtheorem{rem}[theo]{Remark}
\newtheorem{defi}[theo]{Definition}

\newtheorem{assum}[theo]{Assumption}
\newtheorem{prop}[theo]{Proposition}

\newcommand{\norm}[1]{\left\Vert#1\right\Vert}

\title{Functionally-fitted energy-preserving integrators\\ for Poisson systems}

\author{Bin Wang\,
\footnote{School of Mathematical Sciences, Qufu Normal University,
Qufu 273165, P.R. China; Mathematisches Institut, University of
T\"{u}bingen, Auf der Morgenstelle 10, 72076 T\"{u}bingen, Germany.
The research is supported in part by the Alexander von Humboldt
Foundation and by the Natural Science Foundation of Shandong
Province (Outstanding Youth Foundation) under Grant ZR2017JL003.
E-mail:~{\tt wang@na.uni-tuebingen.de} } \and Xinyuan
Wu\thanks{School of Mathematical Sciences, Qufu Normal University,
Qufu 273165, P.R. China; Department of Mathematics, Nanjing
University, Nanjing 210093, P.R. China. The research is supported in
part by the National Natural Science Foundation of China under Grant
11671200. E-mail:~{\tt xywu@nju.edu.cn}} }

\begin{document}
\maketitle

\begin{abstract} In this paper, we present an error analysis of one-stage explicit
extended Runge--Kutta--Nystr\"{o}m integrators for semilinear wave
equations. These equations are analysed  by using spatial
semidiscretizations with periodic boundary conditions in one space
dimension. Optimal second-order convergence is proved  without
requiring Lipschitz continuous and higher regularity of the exact
solution. Moreover, the error analysis is not restricted to the
spectral semidiscretization in space.

\medskip
\no{\bf Keywords:}  error bounds, semilinear wave equations,
extended Runge--Kutta--Nystr\"{o}m  integrators,
 trigonometric integrators

\medskip
\no{\bf MSC:} 65M15, 65P10, 65L70, 65M20

\end{abstract}


\section{Introduction}
In this paper, we are concerned with the error analysis of one-stage
explicit extended Runge--Kutta--Nystr\"{o}m (ERKN) integrators  for
solving the semilinear wave equation
\begin{equation}\label{wave equa}
\begin{array}[c]{ll}
u_{tt}=u_{xx}+u^p,\ \ \ u=u(x,t),\ \ \ t\in[t_0,T].
\end{array}
\end{equation}
We consider  real-valued solutions to \eqref{wave equa} with
$2\pi$-periodic boundary conditions in one space dimension
($x\in\mathbb{T}  = \mathbb{R}/(2\pi \mathbb{Z})$). Denoting by
$H^s$ the Sobolev space $H^s(\mathbb{T})$,  the initial values of
the equation are given by
\begin{equation}\label{initial val}
\begin{array}[c]{ll}
 u(\cdot,t_0)\in H^{s+1},\ \ \  u_t(\cdot,t_0)\in H^{s}
\end{array}
\end{equation}
for $s\geq0.$ We note that the energy is finite in the special case
$ s = 0$.

By a semidiscretization in space, this  equation can be transformed
into  a system of second-order ordinary differential equations
(ODEs) of the form
\begin{equation}
\ddot{y}(t)=M y(t)+f(y(t)), \label{intrprob}%
\end{equation}
where  the matrix $M$ describes   the discretized second spatial
derivative in \eqref{wave equa} and  $f(y)$ denotes  the polynomial
nonlinearity in \eqref{wave equa}.  It is noted that the eigenvalues
of the matrix  $M$  range from order one to the order of the spatial
discretization parameter which is typically large. Thus the spatial
semidiscretization  exhibits a variety of oscillations and  the
solution of \eqref{intrprob} is a high-frequency oscillator, in
general.  In order to
 efficiently solve \eqref{intrprob}, many effective
integrators have been researched  (see, e.g.
\cite{franco2006,hairer2002,Hochbruck2010,iserles2002,wu2013-book},
and the references therein).   The Gautschi-type integrators have
been well derived and analysed in \cite{Hochbruck1999}. Exponential
integrators have been widely developed as an efficient approach and
we refer the reader to
 \cite{Hochbruck2005,Hochbruck2009,Li_Wu(sci2016),wang2017-JCM}
for example.  These methods were also shown to work well for wave
equations in the semilinear case (see, e.g.
\cite{Bao12,Cano13,Cano14,Cohen08,Gauckler17-1}).
 Recently,  a standard form of trigonometric integrators
called as ERKN  integrators   was
 formulated for solving second-order highly oscillatory differential equations  in \cite{wu2010-1}, which can be of arbitrarily high
 order. It is clear that the Gautschi-type integrators of order two yield examples of ERKN integrators.
  Further researches of  these trigonometric integrators for oscillatory ODEs are
referred to \cite{wang-2016,wang2017-ANM,wu2017-JCAM,wang2017-Cal}.
 On basis of  a novel operator-variation-of-constants
formula established in \cite{Wu2015JMAA},   ERKN integrators  are
well developed   for wave equations  (see, e.g.
\cite{Liu_Iserles_Wu(2017-2),JCP-Liu,AML(2017)_Liu_Wu,JCAM(2016)_Wu_Liu,wubook2015}).

It is noted that   the error analysis of trigonometric integrators
for ODEs has been researched by many papers (see, e.g.
\cite{B.garcia1999,Hochbruck1999,Grimm2005
1,Grimm2006,Hochbruck2005,Hochbruck2009,wu2013-ANM}), but has not
sufficient yet because in all these publications,
 the nonlinearity is assumed to be Lipschitz continuous.   There is
also much work about the error analysis of trigonometric integrators
for PDEs. In \cite{Grimm06}, the author analysed the Gautschi-type
trigonometric integrators for sine-Gordon equations whose
nonlinearity is Lipschitz continuous in $H^0$. In \cite{Dong14} and
Chapter IV of \cite{Faou12}, the systems of PDEs were considered in
higher-order Sobolev spaces and second-order error bounds of
trigonometric integrators were shown under  some corresponding
higher regularity assumptions on the exact solution. The author in
\cite{Gauckler15} proved error bounds of trigonometric integrators
for wave equations  without requiring higher regularity of the exact
solution, which was realized  by performing the error analysis  in
two stages. These two-stage arguments have also been used by many
researchers, such as in
\cite{Gauckler11,Koch11,Lubich08,Thalhammer12} for discretizations
of nonlinear Schr\"{o}dinger equations and in
\cite{Gottlieb12,Holden13} for discretizations of equations with
Burgers nonlinearity. In the case of quasilinear wave equations,
error analysis  has been presented recently in
\cite{Gauckler17,Hochbruck17,Lubich17} for different methods.

To our knowledge, the error analysis    of ERKN integrators has not
been   researched yet in the literature for spatial
semidiscretizations of \eqref{wave equa} with initial values of
finite energy. In this paper, we will analyse and present  error
bounds for one-stage explicit ERKN integrators  when applied to a
spectral semidiscretization in space  requiring only that the exact
solution is of finite energy. First, low-order error bounds will be
considered in a higher-order Sobolev space, where the nonlinearity
is, at least locally, Lipschitz continuous. From this low-order
error bound, a suitable regularity of the ERKN integrators will be
deduced, which is used  to overcome the lack of Lipschitz continuity
in lower-order Sobolev spaces. Then higher-order error bounds will
be shown  in these spaces by using the regularity of the ERKN
integrators. It is noted that optimal second-order convergence will
be obtained without requiring Lipschitz continuous and higher
regularity of the exact solution. Moreover,  the approach to the
error analysis  is not restricted to the spectral semidiscretization
in space.

The paper   is organized as follows. We first present  some
preliminaries in Section \ref{sec:prelim}. The main result of
  error bounds is  described     in Section \ref{sec:main
result}. Then the lower-order error bounds in higher-order Sobolev
spaces are proved in Section \ref{sec-prof1} and the higher-order
error bounds in lower-order Sobolev spaces are shown in Section
\ref{sec-prof2}. Error bounds for the spatial semidiscretization by
finite differences are discussed in Section \ref{sec:finite
differences}.
 In Section \ref{sec:experiments},
numerical experiments are  implemented   to demonstrate the
efficiency of the  ERKN  integrators  which  support our theoretical
analysis in this paper. Section \ref{sec:conclusions} is devoted to
the conclusions of the paper.

\section{Preliminaries}\label{sec:prelim}
\subsection{Spectral semidiscretization in space}
Following \cite{Gauckler15,Hairer08},  we choose spectral
collocation for the semidiscretization in space of \eqref{wave
equa}.

Consider the following trigonometric polynomial as an ansatz for the
solution of the nonlinear wave equation   \eqref{wave equa}
\begin{equation}\label{trigo pol}
\begin{array}[c]{ll}
 u_{\mathcal{K}}(x,t_0)=\sum\limits_{j\in\mathcal{K}}
 y_j(t)\mathrm{e}^{\mathrm{i}jx}\quad \ \textmd{with}\quad \
 \mathcal{K}=\{-K,-K+1,\ldots,K-1\},
\end{array}
\end{equation}
where $ y_j(t)$ for $j\in\mathcal{K}$ are the  Fourier coefficients.
By inserting this ansatz into \eqref{wave equa} and evaluating in
the collocation points $x_k = \pi k/K$ with $k \in \mathcal{K}$, we
obtain a system of second-order ODEs
\begin{equation}
\ddot{y}(t)=-\Omega^2 y(t)+f(y(t)), \label{prob}%
\end{equation}
where  $y(t)=(y_j(t))_{j\in \mathcal{K}}\in
\mathbb{C}^{\mathcal{K}}$ is the vector of Fourier coefficients.
Here, $\Omega$ is a nonnegative and diagonal matrix
\begin{equation*}
 \Omega=\textmd{diag}(\omega_j)_{j\in \mathcal{K}}\quad \ \textmd{with}\quad \
\omega_j=|j|,
\end{equation*}
and the nonlinearity $f$ is given by
\begin{equation}\label{f}%
f(y)=\underbrace{y*\ldots* y}_{p\ \textmd{times}}  \quad \
\textmd{with}\quad \ (y*z)_j=\sum\limits_{k+l\equiv j\ \textmd{mod}\
2K}y_kz_l,\ \ j\in \mathcal{K},
\end{equation}
where `$*$' denotes the discrete convolution.
 The initial values $y(t_0)$ and
$\dot{y}(t_0)$ for \eqref{prob} are given respectively by
\begin{equation}\label{ini val spe}%
y_j(t_0)= \sum\limits_{k\in \mathbb{Z}:k\equiv j\ \textmd{mod}\
2\mathcal{K}}u_k(t_0),\ \ \ \ \dot{y}_j(t_0)= \sum\limits_{k\in
\mathbb{Z}:k\equiv j\ \textmd{mod}\ 2K}\dot{u}_k(t_0),\ \ j\in
\mathcal{K},
\end{equation}
where   $u_k(t)$ and $\dot{u}_k(t)$ are denoted the Fourier
coefficients of $u(\cdot, t)$ and $u_t(\cdot, t)$, respectively. If
the initial values $u(\cdot, t_0)$ and $u_t(\cdot, t_0)$ are given
by their Fourier coefficients, we have  a simpler expression:
\begin{equation}\label{ini val-2 spe}%
y_j(t_0)=  u_j(t_0),\ \ \dot{y}_j(t_0)= \dot{u}_j(t_0),\ \ j\in
\mathcal{K}.
\end{equation}

By the variation-of-constants formula, the exact solution of the
semi-discrete system \eqref{prob} can be expressed by
\begin{equation}\label{vari cons}
\begin{array}[c]{ll}
\left(
  \begin{array}{c}
    y(t) \\
     \dot{y}(t) \\
  \end{array}
\right)=R(t-t_0)\left(
  \begin{array}{c}
    y(t_0) \\
     \dot{y}(t_0) \\
  \end{array}
\right)+\displaystyle\int_{t_0}^{t}R(t-\tau)\left(
                                 \begin{array}{c}
                                   0 \\
                                   f(y(\tau)) \\
                                 \end{array}
                               \right)d \tau
\end{array}
\end{equation}
with
 \begin{equation}\label{matr R}%
R(t)=\left(
       \begin{array}{cc}
         \cos(h\Omega) &  t \textmd{sinc}(h\Omega) \\
         -\Omega\sin(h\Omega) & \cos(h\Omega) \\
       \end{array}
     \right),
\end{equation}
 where  $\textmd{sinc}(x)$ is defined by
$\textmd{sinc}(x)=\sin(x)/x.$

\subsection{ERKN integrators} 
ERKN integrators for integrating (\ref{prob}) were first formulated
in \cite{wu2010-1}, and in this paper  we consider only one-stage
explicit scheme.

\begin{defi}
\label{erkn}  (See \cite{wu2010-1}) A one-stage  explicit ERKN
integrator for solving (\ref{prob}) is defined by%
 \begin{equation}
\left\{\begin{array}
[c]{ll}%
y^{n+c_{1}} &
=\phi_{0}(c_{1}^{2}V)y^{n}+hc_{1}\phi_{1}(c_{1}^{2}V)\dot{y}^{n},\\
y^{n+1} & =\phi_{0}(V)y^{n}+h\phi_{1}(V)\dot{y}^{n}+h^{2}
 \bar{b}_{1}(V)f(y^{n+c_{1}}),\\
\dot{y}^{n+1} &
=-h\Omega^2\phi_{1}(V)y^{n}+\phi_{0}(V)\dot{y}^{n}+h\textstyle
b_{1}(V)f(y^{n+c_{1}}),
\end{array}\right.
  \label{methods}%
\end{equation}
where   $h$ is a stepsize, $c_1$ is real constant, $b_{1}(V)$ and
$\bar{b}_{1}(V)$  are matrix-valued functions of $V\equiv
h^{2}\Omega^2$, and $
\phi_{j}(V):=\sum\limits_{k=0}^{\infty}\dfrac{(-1)^{k}V^{k}}{(2k+j)!}$
for $j=0,1,\ldots.$
\end{defi}

It is noted that for $V=h^{2}\Omega^2$, we have
\begin{equation*}
\phi_{0}(V)=\cos(h\Omega),\qquad  \phi_{1}(V)=\textmd{sinc}(h\Omega),\qquad  \phi_{2}(V)=(h\Omega)^{-2}(I-\cos(h\Omega)).%
\end{equation*}
In this paper, we present five practical one-stage explicit ERKN
integrators whose coefficients are displayed  in Table
\ref{praERKN}. By the symmetry conditions and symplecticness
conditions given in \cite{wu2013-book},   it can be verified that
ERKN1 is neither symmetric nor symplectic,   ERKN3 is symmetric and
symplectic, and the others  are symmetric but not symplectic.

\renewcommand\arraystretch{2.3}
\begin{table}$$
\begin{array}{|c|c|c|c|c|c|c|c|}
\hline
\text{Methods} &c_1  &\bar{b}_1(V)   &b_1(V)   & \text{Symmetric}  &  \text{Symplectic}  \\
\hline
\text{ERKN1} & \frac{1}{2} & \phi_{2}(V)& \phi_{0}(V/4)   & \text{Non}  &\text{Non}   \cr
\text{ERKN2} & \frac{1}{2} & \phi_{2}(V) & \phi_{1}(V)   & \text{Symmetric} & \text{Non}  \cr
\text{ERKN3} & \frac{1}{2} & \frac{1}{2} \phi_{1}(V/4) & \phi_{0}(V/4)   & \text{Symmetric} & \text{Symplectic} \cr
\text{ERKN4} & \frac{1}{2} &\frac{1}{2}  \phi^2_{1}(V/4) &\phi_{1}(V/4)\phi_{0}(V/4)   & \text{Symmetric} & \text{Non} \cr
\text{ERKN5} & \frac{1}{2} &\frac{1}{2} \phi_{1}(V)\phi_{1}(V/4) &\phi_{1}(V)\phi_{0}(V/4)   & \text{Symmetric} & \text{Non} \cr
 \hline
\end{array}
$$
\caption{Four one-stage explicit ERKN integrators.} \label{praERKN}
\end{table}

\subsection{Useful results}
In this paper, we measure  the error  by the norm (it has been
considered in \cite{Gauckler15,Hairer08})
\begin{equation}
\norm{y}_s:=\Big(\sum\limits_{ j \in\mathcal{K}}  \langle j\rangle
^{2s}|y_j|^2 \Big)^{1/2}\quad \ \textmd{with}\quad \ \langle
j\rangle=\max(1,|j|)
\label{norm}%
\end{equation}
for $y \in \mathbb{C}^{\mathcal{K}}$, where  $s\in \mathbb{R}$. This
norm is (equivalent to) the Sobolev $H^s$-norm  of the trigonometric
polynomial $\sum\limits_{ j \in\mathcal{K}}
y_j\mathrm{e}^{\mathrm{i}jx}$. It is noted that for this norm, we
have $\norm{y}_{s_1}\leq \norm{ y}_{s_2}$ if $s_1\leq s_2.$

Two important  propositions of the nonlinearity $f(y)$ given in
\cite{Gauckler15} are needed in this paper and we summarize  them as
follows.

%

\begin{prop}\label{pro thm 2}(See \cite{Gauckler15})
 Assume that $\sigma,\sigma'\in \mathbb{R}$ with  $\sigma'\geq |\sigma|$ and
$\sigma'\geq 1$. If $\norm{y}_{\sigma'}\leq M$ and
$\norm{z}_{\sigma'}\leq M,$ then we have
\begin{eqnarray}
   \norm{f(y)-f(z)}_{\sigma} &\leq& C \norm{y-z}_{\sigma},   \label{norm pro2-1} 
 \\ \norm{f(y)}_{\sigma'} &\leq& C \label{norm pro2-2}
\end{eqnarray}
with a constant $C$ depending only on $M, |\sigma|, \sigma',$ and
$p$.
\end{prop}

\begin{prop}\label{pro thm 3}(See \cite{Gauckler15})
 Let  $s\geq0,$ and assume that for $y:[t_0,t_1]\rightarrow
\mathbb{C}^{\mathcal{K}}$
$$\norm{y(t)}_{s+1}\leq M,\qquad \norm{\dot{y}(t)}_{s}\leq M,\qquad \textmd{for} \qquad t_0\leq t
\leq t_1.$$  It then holds  that
\begin{equation}
\norm{ \frac{d}{d t}f(y(t))}_{s}\leq C,
\label{norm pro3-1}%
\end{equation}
where the constant $C$ depends on $M, s,$ and $p$. If, in addition,
if    $\norm{\ddot{y}(t)}_{s-1}\leq M$  for $t_0\leq t \leq t_1,$
then
\begin{equation}
\norm{ \frac{d^2}{d t^2}f(y(t))}_{s-1} \leq C
\label{norm pro3-2}%
\end{equation}
with a constant $C$ depending on $M, s,$ and $p$.
\end{prop}

\section{Main result}\label{sec:main result}

Before presenting the error bounds, we need the following
assumptions of the coefficients of the ERKN integrators. Similar
assumptions  on the filter functions of some
 trigonometric methods have
been considered in \cite{Gauckler15}.

\begin{assum}\label{thm ass1}
It is  assumed that there exists a constant $c$ such that
\begin{eqnarray}
   |\bar{b}_1(\xi^2)|\leq& c \xi^{\beta}  \qquad &\textmd{if} \qquad -1\leq\beta\leq 0, \label{ass1} 
 \\   |1/2 \textmd{sinc}^2(\xi/2)-\bar{b}_1(\xi^2)|\leq&c\xi^{\beta} \qquad &\textmd{if} \qquad 0 <\beta\leq 1,\label{ass2}
 \\   |1-b_1(\xi^2)|\leq&c \xi^{1+\beta} \label{ass3}
\end{eqnarray}
for  $-1\leq\beta\leq 1$ and all $\xi = h\omega_j$ with $j
\in\mathcal{K}$ and $\omega_j\neq0$. Moreover, we assume that
$c_1=\frac{1}{2}$ for the   ERKN integrator \eqref{methods}.
\end{assum}

It can be  verified  that all the  ERKN integrators displayed in
Table \ref{praERKN} satisfy this assumption uniformly for
$-1\leq\beta\leq 1$ and $h>0$.

Under this assumption, we have the following property with respect
to the norm \eqref{norm}.
\begin{prop}\label{pro norm }
Under the conditions of Assumption \ref{thm ass1} and for  $s\in
\mathbb{R}$, it holds that
\begin{equation*}\begin{array}[c]{ll}%
&\norm{\bar{b}_1(V)y}_{s-\beta}\leq ch^{\beta}\norm{ y}_{s},\qquad
\qquad \qquad \qquad \quad \
\textmd{if} \quad -1\leq\beta\leq 0,\\
&\norm{\big(1/2\textmd{sinc}^2(h\Omega/2)-\bar{b}
_1(V)\big)y}_{s-\beta}\leq ch^{\beta}\norm{ y}_{s},\quad
\textmd{if}  \quad 0 <\beta\leq 1,\\
&\norm{y-b_1(V)y}_{s-\beta}\leq ch^{1+\beta}\norm{ y}_{s+1}.
\end{array}
\end{equation*}

\end{prop}
\begin{proof} For the first one, according to the definition of  the norm \eqref{norm}, one has
\begin{equation*}
\begin{array}
[c]{ll}%
&\norm{\bar{b}_1(V)y}_{s-\beta}=\Big(\sum\limits_{ j \in\mathcal{K}}
\langle
j\rangle ^{2s-2\beta}|\bar{b}_1(h^2\omega^2_j) y_j|^2 \Big)^{1/2}\\
\leq&c\Big(\sum\limits_{ j \in\mathcal{K}} \langle j\rangle
^{2s-2\beta}| (h\omega_j)^{\beta} y_j|^2 \Big)^{1/2}\leq
c\Big(\sum\limits_{ j \in\mathcal{K}} \langle
j\rangle ^{2s}|  h ^{\beta} y_j|^2 \Big)^{1/2}\\
\leq& ch ^{\beta}\Big(\sum\limits_{ j \in\mathcal{K}} \langle
j\rangle ^{2s}|   y_j|^2 \Big)^{1/2}=ch^{\beta}\norm{ y}_{s}.\\
\end{array}
\end{equation*}
Other results can be obtained in a similar way. \end{proof}

The main result of this paper is given by the following theorem.
 \begin{theo}\label{thm main-re}
Assume that  the following  three conditions are satisfied:
\begin{itemize}
\item Let $c\geq 1$ and $s\geq 0$.

\item  The exact solution
$(y(t), \dot{y}(t))$ of the spatial semidiscretization \eqref{prob}
   satisfies
\begin{equation}\norm{y(t)}_{s+1}+ \norm{\dot{y}(t)}_{s}\leq M \qquad \textmd{for} \qquad 0\leq
t-t_0 \leq T.\label{thm con1}%
\end{equation}

\item Assumption \ref{thm ass1} is true with constant $c$ for $\beta= 0$
and $\beta= \alpha$ with   $-1 \leq\alpha\leq 1$.

\end{itemize}
Then, there exists $h_0 > 0$ such that for  $0<h \leq h_0$, the
error bound for the numerical solution $(y^n, \dot{y}^n)$ given by
the ERKN integrator \eqref{methods} is
$$\norm{y(t_n)-y^n}_{s+1-\alpha}+
\norm{\dot{y}(t_n)-\dot{y}^n}_{s-\alpha}\leq Ch^{1+\alpha} \qquad
\textmd{for} \qquad 0\leq t_n-t_0=nh \leq T,$$  where the constants
$C$ and $h_0$ depend only on $M$ and $s$ from \eqref{thm con1}, the
power $p$ of the nonlinearity in \eqref{wave equa}, the final time
$T$, and the constant $c$ in Assumption \ref{thm ass1}.
\end{theo}

\begin{rem}
It is noted that if we consider $s = 0$ in this theorem, the
condition \eqref{thm con1} presents a finite energy assumption on
the solution of the nonlinear wave equation \eqref{wave equa} and
its spatial semidiscretization \eqref{prob}. By choosing $\alpha =
1$ and $\alpha = 0$, we get $\norm{y(t_n)-y^n}_{0}\leq Ch^{2}$ and
$\norm{\dot{y}(t_n)-\dot{y}^n}_{0}\leq Ch$, respectively. This leads
to
 a second-order error bound for $y$ in
$H^0=L^2$ and a first-order error bound for $\dot{y}$ in $L^2$.
\end{rem}

By the two-stage arguments used in
\cite{Gauckler15,Gauckler11,Koch11,Lubich08,Thalhammer12}, the proof
of  Theorem \ref{thm main-re} will be divided into two parts. We
will first  show   the proof of the lower-order error bounds in
higher-order Sobolev spaces (i.e., $-1 \leq\alpha\leq 0$) in Section
\ref{sec-prof1}, and then present the proof of the higher-order
error bounds in lower-order Sobolev spaces (i.e.,
 $0<\alpha\leq 1$) in Section \ref{sec-prof2}.

\section{Proof of the lower-order error bounds in higher-order
Sobolev spaces}\label{sec-prof1} Throughout the proof,  we assume
that $0<h \leq 1$ and use the norm
$$|||(y,\dot{y})|||_{\sigma}=(\norm{y}^2_{\sigma+1}+\norm{\dot{y}}^2_{\sigma})^{1/2}$$
on $H^{\sigma+1}\times H^{\sigma}$ for  $\sigma\in \mathbb{R}$.  The
exact solution \eqref{vari cons} of the system \eqref{prob}  is
denoted by $(y(t), \dot{y}(t))$ and  the numerical results of ERKN
integrators \eqref{methods} are denoted by $(y^0, \dot{y}^0),(y^1,
\dot{y}^1),\ldots$.

\subsection{Regularity over one time step} We first show  the preservation of regularity of
the ERKN integrator \eqref{methods} over one time step.

\begin{lem}\label{lem-regularity}
Let $s\geq0$ and  $-1 \leq\alpha\leq 0$. It is assumed that
Assumption \ref{thm ass1} holds for $\beta=\alpha$   with constant
$c$, and $|||(y^0,\dot{y}^0)|||_{s}\leq M,$ then for the solution
given by the  ERKN integrator \eqref{methods}, we have
$$|||(y^1,\dot{y}^1)|||_{s}  \leq C,$$ where $C$ depends only on $M,
s, p,$ and $c$.
\end{lem}
\begin{proof}
 From  the definition of the integrator \eqref{methods},
 it follows that
$$\norm{y^{\frac{1}{2}}}_{s+1} \leq \norm{\cos\big(\frac{1}{2}h\Omega\big)y^{0}}_{s+1}+
h\frac{1}{2}\norm{\textmd{sinc}\big(\frac{1}{2}h\Omega\big)\dot{y}^{0}}_{s+1}.$$
Using $\textmd{sinc}(0)=1\leq h^{-1}$ and the bound
$|\textmd{sinc}(\xi)| \leq \xi^{-1}$ for $\xi>0$, we obtain
$$\norm{y^{\frac{1}{2}}}_{s+1} \leq \norm{ y^{0}}_{s+1}+
 \norm{ \dot{y}^{0}}_{s}\leq 2M,$$
which gives
\begin{equation}\norm{f(y^{\frac{1}{2}})}_{s+1} \leq C\label{fhaf1 bound}%
\end{equation}
by considering \eqref{norm pro2-2} with $\sigma'=s+1.$

 On the basis of  the scheme of  the integrator \eqref{methods}
again, it can be obtained that
$$\norm{y^1}_{s+1} \leq \norm{\cos( h\Omega)y^{0}}_{s+1}+
h \norm{\textmd{sinc}(
h\Omega)\dot{y}^{0}}_{s+1}+h^2\norm{\bar{b}_1( h\Omega)f(y^{
\frac{1}{2}})}_{s+1}.$$ In the light of the fact that $-1
\leq\alpha\leq 0 $ and the bound \eqref{ass1} of $\bar{b}_1$, one
arrives
$$\norm{y^1}_{s+1} \leq \norm{ y^{0}}_{s+1}+
 \norm{ \dot{y}^{0}}_{s}+h^{2+\alpha}\norm{f(y^{ \frac{1}{2}})}_{s+1+\alpha}
 \leq \norm{ y^{0}}_{s+1}+
 \norm{ \dot{y}^{0}}_{s}+h^{2+\alpha}\norm{f(y^{ \frac{1}{2}})}_{s+1}.$$
It follows from   \eqref{fhaf1 bound} that $\norm{y^1}_{s+1}  \leq
C.$ Similarly, we can obtain  $\norm{\dot{y}^1}_{s}  \leq C$, and
then the proof is complete.
\end{proof}

\subsection{Local error bound}
In what follows, we consider the local error of the ERKN integrator
\eqref{methods}.

\begin{lem}\label{lem-local error}
(Local error in $H^{s+1-\alpha}\times H^{s-\alpha}$ for $-1
\leq\alpha\leq 0$.) Under the conditions of Lemma
\ref{lem-regularity}, if $$|||(y(\tau),\dot{y}(\tau))|||_{s}\leq M$$
for $t_0\leq \tau \leq t_1,$ then  we have
$$|||(y(t_1),\dot{y}(t_1))-(y^1,\dot{y}^1)|||_{s-\alpha}\leq Ch^{2+\alpha},$$
where the constant $C$ depends only on $M, s, p,$ and $c$.
\end{lem}

\begin{proof}Throughout the proof,   denote by $C$ a generic constant which depends on
$M, s, p,$ and $c$.

\vskip2mm \textbf{(I) } We first analyse the local error of
$y(t_1)-y^1$. According to the variation-of-constants formula
\eqref{vari cons} and the scheme of the integrator \eqref{methods},
we obtain
\begin{equation*}
\begin{array}[c]{ll}
y(t_1)-y^1=
\displaystyle\int_{t_0}^{t_1}(t_1-\tau)\textmd{sinc}((t_1-\tau)\Omega)f(y(\tau))d
\tau-h^{2}
 \bar{b}_{1}(V)f(y^{ \frac{1}{2}}).
\end{array}
\end{equation*}
For $\xi>0$ and $-1 \leq\alpha\leq 0$, it can be verified  that
$$|\textmd{sinc}(\xi)| =\frac{|\sin(\xi)|}{\xi} \leq\left\{
\begin{aligned}
 &1\leq\xi^{\alpha},\qquad \ \ \textmd{if}\ \ \ 0<\xi\leq1,\\
  &\xi^{-1}\leq\xi^{\alpha},\quad \ \ \ \textmd{if}\ \ \  \xi>1, \end{aligned}\right.
 $$ which yields  $|\textmd{sinc}(\xi)|  \leq
 \xi^{\alpha}.$ Moreover, it is clear that
 $h^{\alpha}\geq1$ for $-1\leq\alpha\leq 0$. By these results and \eqref{ass1}, one gets
$$\norm{y(t_1)-y^1}_{s+1-\alpha} \leq h^{2+\alpha} \sup_{t_0\leq \tau \leq t_1}\norm{f(y(\tau))}_{s+1}+
c   h^{2+\alpha} \norm{ f(y^{ \frac{1}{2}})}_{s+1}.$$ According to
\eqref{norm pro2-2}, it is true that $\norm{f(y(\tau))}_{s+1}\leq
C,$ which leads to $\norm{y(t_1)-y^1}_{s+1-\alpha} \leq C
h^{2+\alpha} $ by considering \eqref{fhaf1 bound}.

\vskip2mm \textbf{(II)}  We now   discuss the local error of
$\dot{y}(t_1)-\dot{y}^1$.
 It follows from  \eqref{vari cons} and
\eqref{methods} that
\begin{equation*}\label{err-dy1}
\begin{array}[c]{ll}
\dot{y}(t_1)-\dot{y}^1=
\displaystyle\int_{t_0}^{t_1}\cos((t_1-\tau)\Omega)f(y(\tau))d
\tau-h
 b_{1}(V)f(y^{ \frac{1}{2}}),
\end{array}
\end{equation*}
which can be split as
\begin{eqnarray}
   &\ & \dot{y}(t_1)- \dot{y}^1  \nonumber \\
   &= & \displaystyle\int_{t_0}^{t_1}[\cos((t_1-\tau)\Omega)-I_d]f(y(\tau))d \tau  \label{dy1 er 1}\\
    &+ & \displaystyle\int_{t_0}^{t_1} f(y(\tau))d \tau-h
  f\big(y \big(\frac{t_0+t_1}{2}\big)\big) \label{dy1 er 2}\\
    &+ & h
    f\big(y \big(\frac{t_0+t_1}{2}\big)\big)  -h
 f(y^{ \frac{1}{2}})\label{dy1 er 3}\\
     &+ & h
  (I_d- b_{1}(V))
f(y^{ \frac{1}{2}}).\label{dy1 er 4}
\end{eqnarray}

 $\bullet$ Bound of \eqref{dy1 er 1}.
For $\xi>0$ and $-1 \leq\alpha\leq 0$,   it is easy to obtain that
$$|\cos(\xi)-1|=2|\sin(\xi/2)|^2 \leq\left\{
\begin{aligned}
 &2 (\xi/2)^2 \leq 2\xi^{1+\alpha},\ \ \ \textmd{if}\ \ \ 0<\xi\leq1,\\
  &2\leq2\xi^{1+\alpha},\qquad \quad \ \textmd{if}\ \ \  \xi>1, \end{aligned}\right.
 $$
 which yields that  $|\cos(\xi)-1| \leq 2 \xi^{1+\alpha}.$
 On noticing  \eqref{norm pro2-2} with $\sigma'=s+1,$ one arrives at
\begin{equation*}
\begin{array}[c]{ll}
&\norm{\displaystyle\int_{t_0}^{t_1}[\cos((t_1-\tau)\Omega)-I_d]f(y(\tau))d
\tau}_{s-\alpha}\\
\leq& 2  h^{1+\alpha}\displaystyle\int_{t_0}^{t_1}
\norm{f(y(\tau))}_{s+1}d \tau
 \leq 2  h^{1+\alpha}\displaystyle\int_{t_0}^{t_1} C d \tau
 \leq C h^{2+\alpha}.
\end{array}
\end{equation*}

 $\bullet$
  Bound of \eqref{dy1 er 2}. Since $1 \leq \xi^{1+\alpha} + \xi^{ \alpha}$ for  $\xi >
0$, we rewrite \eqref{dy1 er 2} as
\begin{equation*}
\begin{array}[c]{ll}
&\norm{\displaystyle\int_{t_0}^{t_1} f(y(\tau))d \tau-h
  f\big(y \big(\frac{t_0+t_1}{2}\big)\big)}_{s-\alpha} \\
  \leq& h^{1+\alpha}\norm{\displaystyle\int_{t_0}^{t_1} f(y(\tau))d \tau-h
  f\big(y \big(\frac{t_0+t_1}{2}\big)\big)}_{s+1}  +  h^{\alpha}\norm{\displaystyle\int_{t_0}^{t_1} f(y(\tau))d \tau-h
  f\big(y \big(\frac{t_0+t_1}{2}\big)\big)}_{s}.
\end{array}
\end{equation*}
From   \eqref{norm pro2-2} with $\sigma'=s+1,$ it follows that
$$\norm{\displaystyle\int_{t_0}^{t_1} f(y(\tau))d \tau
-hf\big(y \big(\frac{t_0+t_1}{2}\big)\big)}_{s+1} \leq
\displaystyle\int_{t_0}^{t_1} C d \tau+Ch\leq Ch.$$ For an estimate
in the norm $\norm{\cdot}_{s}$, it is remarked that \eqref{dy1 er 2}
is the quadrature error of the  mit-point rule. With its first-order
Peano kernel $K_1(\tau)$ and by the Peano kernel theorem,  one
arrives
\begin{equation*}
\begin{array}[c]{ll}&\norm{\displaystyle\int_{t_0}^{t_1} f(y(\tau))d \tau-h
f\big(y \big(\frac{t_0+t_1}{2}\big)\big)}_{s}
=h^2\norm{\displaystyle\int_{t_0}^{t_1} K_1(\tau) \frac{d}{dt}
f(y(t_0+\tau h))d \tau }_{s} \leq Ch^2,\end{array}
\end{equation*}
where we have used \eqref{norm pro3-1}. Thus, it is true that
\begin{equation}\label{err-f01}\norm{\displaystyle\int_{t_0}^{t_1} f(y(\tau))d \tau-h
f\big(y \big(\frac{t_0+t_1}{2}\big)\big)}_{s-\alpha}\leq C
h^{2+\alpha}.\end{equation}

  $\bullet$  Bound of \eqref{dy1 er 3}. By  \eqref{norm pro2-1} with $\sigma=s-\alpha$, we have
$$\norm{h   f\big(y \big(\frac{t_0+t_1}{2}\big)\big)   -hf(y^{ \frac{1}{2}} ) }_{s-\alpha}\leq C h\norm{ y
\big(\frac{t_0+t_1}{2}\big)- y^{ \frac{1}{2}} }_{s-\alpha}.$$
According to  \eqref{vari cons} and  \eqref{methods}, it is obtained
that
\begin{equation}\label{err-y1}
\begin{array}[c]{ll}
y(\frac{t_0+t_1}{2})-y^{ \frac{1}{2}}=
\displaystyle\int_{t_0}^{\frac{t_0+t_1}{2}}\big(\frac{t_0+t_1}{2}-\tau\big)
\textmd{sinc}\big((\frac{t_0+t_1}{2}-\tau)\Omega\big)f(y(\tau))d\tau.
\end{array}
\end{equation}
 In a similar way to that for  the first part of this proof, one gets
$$\norm{y\big(\frac{t_0+t_1}{2}\big)-y^{ \frac{1}{2}}}_{s+1-\alpha} \leq h^{2+\alpha}
\sup_{t_0\leq \tau \leq \frac{t_0+t_1}{2}}\norm{f(y(\tau))}_{s+1}
\leq C h^{2+\alpha}.$$  Thus,  it is true that
$$\norm{y \big(\frac{t_0+t_1}{2}\big)- y^{ \frac{1}{2}} }_{s-\alpha}\leq
\norm{y \big(\frac{t_0+t_1}{2}\big)- y^{ \frac{1}{2}}
}_{s+1-\alpha}\leq C h^{2+\alpha},$$ which yields
$$\norm{hf\big(y \big(\frac{t_0+t_1}{2}\big)\big)-h f(y^{ \frac{1}{2}}) }_{s-\alpha}  \leq C h^{3+\alpha}.$$

   $\bullet$  Bound of \eqref{dy1 er 4}. According to  \eqref{fhaf1 bound} and the bound   \eqref{ass3}, we get
$$\norm{ h
  (I_d- b_{1}(V))
f(y^{ \frac{1}{2}})}_{s-\alpha}  \leq C h^{2+\alpha} \norm{  f(y^{
\frac{1}{2}})}_{s+1} \leq C h^{2+\alpha}.$$

All these estimates of the single terms in \eqref{dy1 er
1}-\eqref{dy1 er 4}  together imply $\norm{ \dot{y}(t_1)-\dot{y}^1
}_{s-\alpha}  \leq C h^{2+\alpha}.$ The proof is complete.
\end{proof}

\subsection{Stability} This subsection studies the stability of the ERKN integrator
\eqref{methods}.
\begin{lem}\label{lem-stabi}
 (Stability  in $H^{s+1-\alpha}\times H^{s-\alpha}$ for $-1
\leq\alpha\leq 0$.) Under the conditions of Lemma
\ref{lem-regularity},   if we consider the ERKN integrator
\eqref{methods} with different initial values $(y_0,\dot{y}_0)$ and
$(z_0,\dot{z}_0)$ satisfying
$$|||(y_0,\dot{y}_0)|||_{s}\leq M\qquad \textmd{and} \qquad |||(z_0,\dot{z}_0)|||_{s}\leq M,$$
then one gets
$$|||(y^1,\dot{y}^1)-(z^1,\dot{z}^1)|||_{s-\alpha}\leq (1+Ch) |||(y^0,\dot{y}^0)-(z^0,\dot{z}^0)|||_{s-\alpha},$$
where the constant $C$ depends only on $M, s, p,$ and $c$.
\end{lem}
\begin{proof}
For $ (w^{\intercal},\dot{w}^{\intercal})^{\intercal}=R(h)
(v^{\intercal},\dot{v}^{\intercal})^{\intercal}, $ it has been shown
in \cite{Gauckler15} that
\begin{equation*}\label{Rh bound}
|||(w,\dot{w})|||_{\sigma}=|||(v,\dot{v})+h(\tilde{v},0)|||_{\sigma},
\end{equation*}
where $\tilde{v}\in \mathbb{C}^{\mathcal{K}} $  and
$\tilde{v}_j=\delta_{j,0}\dot{v}_j$  with the Kronecker $\delta$.
 Using this result and the scheme of ERKN integrators
\eqref{methods}, we obtain
\begin{eqnarray}
  |||(y^1,\dot{y}^1)-(z^1,\dot{z}^1)|||_{s-\alpha} &\leq&  |||(y^0,\dot{y}^0)-(z^0,\dot{z}^0)|||_{s-\alpha}  \nonumber \\
   &+ & h |\dot{y}^0_0-\dot{z}^0_0|  \label{sta er 1}\\
    &+ & h^2\norm{\bar{b}_1(V)  \big(f(y^{\frac{1}{2}}) -f(z^{\frac{1}{2}})\big) }_{s+1-\alpha} \label{sta er 2}\\
    &+ & h \norm{b_1(V)  \big(f(y^{\frac{1}{2}}) -f(z^{\frac{1}{2}})\big) }_{s-\alpha}. \label{sta er 3}
\end{eqnarray}

 $\bullet$   For \eqref{sta er 1}, it is clear that $h
|\dot{y}^0_0-\dot{z}^0_0|\leq h
\norm{\dot{y}^0_0-\dot{z}^0_0}_{s-\alpha}. $

 $\bullet$  For \eqref{sta er 2},  considering the bound \eqref{ass1} of
$\bar{b}_1$ and the estimate \eqref{norm pro2-2} with
$\sigma=\sigma'=s+1$  leads to
$$h^2\norm{\bar{b}_1(V)  \big(f(y^{\frac{1}{2}}) -f(z^{\frac{1}{2}})\big) }_{s+1-\alpha}  \leq
C h^{2+\alpha}\norm{  f(y^{\frac{1}{2}}) -f(z^{\frac{1}{2}}) }_{s+1}
\leq C h^{2+\alpha} \norm{ y^{\frac{1}{2}}-z^{\frac{1}{2}}
}_{s+1}.$$ Using the scheme of ERKN integrators \eqref{methods}
again,  we  deduce that
$$y^{\frac{1}{2}}-z^{\frac{1}{2}}=\cos\big(\frac{1}{2}h\Omega\big) (y^0-z^0)+h\frac{1}{2} \textmd{sinc}\big(\frac{1}{2}h\Omega\big)(\dot{y}^0-\dot{z}^0),$$
which gives $$\norm{ y^{\frac{1}{2}}-z^{\frac{1}{2}} }_{s+1}\leq
\norm{ y^0-z^0}_{s+1}+\norm{\dot{y}^0-\dot{z}^0}_{s}.$$ This yields
$$h^2\norm{\bar{b}_1(V)  \big(f(y^{\frac{1}{2}})
-f(z^{\frac{1}{2}})\big) }_{s+1-\alpha}  \leq C h^{2+\alpha}
 \norm{ y^0-z^0}_{s+1}+ C h^{2+\alpha}
\norm{\dot{y}^0-\dot{z}^0}_{s}.$$

  $\bullet$ For \eqref{sta er 3}, since $|b_1(\xi)|\leq1+c\xi^{1+\alpha}$ from
\eqref{ass3}, we obtain
\begin{equation*}
\begin{array}[c]{ll}&h \norm{b_1(V)   f(y^{\frac{1}{2}}) -f(z^{\frac{1}{2}})
}_{s-\alpha}\\
\leq&   h  \norm{  f(y^{\frac{1}{2}}) -f(z^{\frac{1}{2}})
}_{s-\alpha} +ch^{2+\alpha} \norm{ f(y^{\frac{1}{2}})
-f(z^{\frac{1}{2}})  }_{s+1} \\
\leq & C h  \norm{  y^{\frac{1}{2}}-z^{\frac{1}{2}} }_{s-\alpha}
+Ch^{2+\alpha} \norm{ y^{\frac{1}{2}}-z^{\frac{1}{2}}  }_{s+1 }\\
\leq  &C (h+h^{2+\alpha})  \norm{  y^0-z^0}_{s+1} +  C
(h+h^{2+\alpha}) \norm{\dot{y}^0-\dot{z}^0}_{s}.
\end{array}
\end{equation*}

These estimates of \eqref{sta er 1}-\eqref{sta er 3} as well as
$\alpha \leq0$ complete the proof.
\end{proof}

\subsection{Proof  of Theorem \ref{thm main-re} for $-1\leq\alpha\leq0$}
  Based on the above three lemmas, we are now  in
a position  to present the proof of Theorem \ref{thm main-re} for
$-1\leq\alpha\leq0$.

\begin{proof}
 \textbf{(I)} We first present the proof for  the case $\alpha=0$. Let $C_1$ and $C_2$  be the constants of Lemmas \ref{lem-local error} and
\ref{lem-stabi} with  $\alpha=0$, respectively. It is noted that
Lemma \ref{lem-stabi} is considered with $2M$  instead of $M$. Let
$h_0 = M/(C_1T e^{C_2T} )$ and  we show  by induction on $n$ that
for   $h\leq h_0$
\begin{equation}\label{er ma re 1}
|||(y^{n},\dot{y}^{n})-(y (t_n),\dot{y}(t_n))|||_{s}\leq C_1
e^{C_2nh}nh^2\end{equation} as long as $t_{n}-t_0=nh\leq T.$

 Firstly, it is clear that
$|||(y^{0},\dot{y}^{0})-(y (t_0),\dot{y}(t_0))|||_{s}=0\leq C_1.$
Assume that the  result \eqref{er ma re 1} is true for
$n=0,1,\ldots,m-1$, which means that
$$|||(y^{m-1},\dot{y}^{m-1})-(y (t_{m-1}),\dot{y}(t_{m-1}))|||_{s}\leq C_1 e^{C_2(m-1)h}(m-1)h^2.$$
This gives
$$|||(y^{m-1},\dot{y}^{m-1})|||_{s}\leq M+C_1 e^{C_2(m-1)h}(m-1)h^2\leq M+C_1 e^{C_2T}Th\leq2M$$
as long as $t_{m-1}-t_0=(m-1)h\leq T.$

Denoting by $\mathcal{E}$ one time step with the ERKN integrator
\eqref{methods}, we obtain
\begin{eqnarray}
 |||(y^{m},\dot{y}^{m})-(y (t_m),\dot{y}(t_m))|||_{s}&=&   |||\mathcal{E}(y^{m-1},\dot{y}^{m-1})-(y (t_m),\dot{y}(t_m))|||_{s} \nonumber \\
   &\leq&  |||\mathcal{E}(y^{m-1},\dot{y}^{m-1})-\mathcal{E}(y (t_{m-1}),\dot{y}(t_{m-1}))|||_{s}   \label{er pr 1}\\
    &+ &  |||\mathcal{E}(y (t_{m-1}),\dot{y}(t_{m-1}))-(y (t_m),\dot{y}(t_m))|||_{s}.  \label{er pr 2}
\end{eqnarray}
 According to Lemma \ref{lem-stabi}, \eqref{er pr 1} admits  the bound
\begin{equation*}
\begin{array}[c]{ll}&|||\mathcal{E}(y^{m-1},\dot{y}^{m-1})-\mathcal{E}(y
(t_{m-1}),\dot{y}(t_{m-1}))|||_{s}\\
\leq& (1+C_2h) |||(y^{m-1},\dot{y}^{m-1})-(y
(t_{m-1}),\dot{y}(t_{m-1}))|||_{s } \leq (1+C_2h)  C_1
e^{C_2(m-1)h}(m-1)h^2.
\end{array}
\end{equation*}
 With regard to  \eqref{er pr 2}, it follows from Lemma
\ref{lem-local error} that
\begin{equation*}
\begin{array}[c]{ll}&|||\mathcal{E}(y (t_{m-1}),\dot{y}(t_{m-1}))-(y (t_m),\dot{y}(t_m))|||_{s}
\leq C_1h^{2}.\\
\end{array}
\end{equation*}
 Thus,  we get
\begin{equation*}
\begin{array}[c]{ll} & |||(y^{m},\dot{y}^{m})-(y (t_m),\dot{y}(t_m))|||_{s}
\leq  (1+C_2h)  C_1 e^{C_2(m-1)h}(m-1)h^2+C_1h^{2}.
\end{array}
\end{equation*}
In what follows, we prove that
\begin{equation}\label{ind-fact}
\begin{array}[c]{ll}   (1+C_2h)  C_1 e^{C_2(m-1)h}(m-1)h^2+C_1h^{2}
\leq     C_1 e^{C_2mh}mh^2.
\end{array}
\end{equation}
By the Taylor expansion of  $e^{C_2(m-1)h}$, we get
\begin{equation*}
\begin{array}[c]{ll}  \ \ \ & (1+C_2h)  C_1
e^{C_2(m-1)h}(m-1)h^2+C_1h^{2}\\
=& (1+C_2h)  C_1
\sum\limits_{k=0}^{\infty}\dfrac{(C_2(m-1)h)^{k}}{k!}
 (m-1)h^2+C_1h^{2}\\
 =&C_1\sum\limits_{k=0}^{\infty}\dfrac{1}{k!}C_2^k (m-1)^{k+1}h^{k+2}+
 C_1\sum\limits_{k=0}^{\infty}\dfrac{1}{k!}C_2^{k+1} (m-1)^{k+1}h^{k+3}+C_1h^{2}\\
  =&C_1mh^2 +C_1\sum\limits_{k=1}^{\infty}\big((m-1)^{k+1}+k(m-1)^{k}\big)\dfrac{1}{k!}C_2^k h^{k+2}.\\
\end{array}
\end{equation*}
 In a similar way, one has
$  C_1 e^{C_2mh}mh^2
  =C_1mh^2 +C_1\sum\limits_{k=1}^{\infty} m^{k+1}\dfrac{1}{k!}C_2^k h^{k+2}.
$ Comparing these two results and  using  the following fact
\begin{equation*}
\begin{array}[c]{ll}   m^{k+1}= (m-1+1)^{k+1}\geq
(m-1)^{k+1}+k(m-1)^{k} \quad  \textmd{for}  \quad m\geq1,\ \ k\geq1,
\end{array}
\end{equation*}
we  conclude  that   \eqref{ind-fact} is true.
 Therefore, \eqref{er ma re 1} holds and hence
\begin{equation*}
|||(y^{n},\dot{y}^{n})-(y (t_n),\dot{y}(t_n))|||_{s}\leq  C_1T
e^{C_2T}h\leq  Ch,\end{equation*} which proves
 the statement of Theorem \ref{thm main-re}
for $\alpha=0$.

\textbf{(II)}  We then  consider the case $-1\leq\alpha < 0$. Let
$h_0$ be as above and let further $C_1$ and $C_2$ be as above but
for the new $\alpha$ instead of $\alpha=0$. In what follows, we
prove, by induction on $n$ that
\begin{equation}\label{er ma re 2}
|||(y^{n},\dot{y}^{n})-(y (t_n),\dot{y}(t_n))|||_{s-\alpha}\leq C_1
e^{C_2nh}nh^{2+\alpha}\end{equation} as long as $t_{n}-t_0=nh\leq
T.$

It is clear that this holds for $n=0$.     From the above proof for
the case $\alpha=0$, it follows  that
$|||(y^{n-1},\dot{y}^{n-1})|||_{s}\leq2M$ as long as
$t_{n-1}-t_0=(n-1)h\leq T.$ This allows us to apply Lemmas
\ref{lem-local error}  and \ref{lem-stabi}   to \eqref{er ma re 2},
which gives
\begin{equation*}
\begin{array}[c]{ll} &|||(y^{n},\dot{y}^{n})-(y
(t_n),\dot{y}(t_n))|||_{s-\alpha} \leq
|||\mathcal{E}(y^{n-1},\dot{y}^{n-1})-\mathcal{E}(y
(t_{n-1}),\dot{y}(t_{n-1}))|||_{s-\alpha}\\
&+   |||\mathcal{E}(y
(t_{n-1}),\dot{y}(t_{n-1}))-(y (t_n),\dot{y}(t_n))|||_{s-\alpha}\\
\leq&  (1+C_2h)||| (y^{n-1},\dot{y}^{n-1})- (y
(t_{n-1}),\dot{y}(t_{n-1}))|||_{s-\alpha}+   C_1h^{2+\alpha}\\
\leq&  (1+C_2h) C_1 e^{C_2(n-1)h}(n-1)h^{2+\alpha}+ C_1h^{2+\alpha}
\leq C_1 e^{C_2nh}nh^{2+\alpha}.
\end{array}
\end{equation*}
This means that \eqref{er ma re 2} is true and thus one has
\begin{equation*}
|||(y^{n},\dot{y}^{n})-(y (t_n),\dot{y}(t_n))|||_{s-\alpha}\leq C_1
T e^{C_2 T}h^{1+\alpha}\leq Ch^{1+\alpha}.\end{equation*}
\end{proof}

\begin{rem}\label{re regularity num}
From the above proof  for $\alpha = 0$, it follows that the
numerical solutions are bounded in $H^{s+1} \times H^s$
\begin{equation}\label{regularity num}
|||(y^{n},\dot{y}^{n})|||_{s}\leq 2M \qquad \textmd{for}\qquad 0\leq
t_n-t_0=nh\leq T.
\end{equation}   This regularity of the numerical solution is essential
for the proof of Theorem \ref{thm main-re} for $0 <\alpha \leq 1$ in
the next section.
\end{rem}

\section{Proof of the higher-order error bounds in lower-order
Sobolev spaces}\label{sec-prof2}

In order to prove Theorem \ref{thm main-re} in lower-order Sobolev
spaces, we first need to present and prove the following three
lemmas.

\begin{lem}\label{lem-regularity-new}
 Letting $s\geq0$ and  $0<\alpha\leq 1,$ and assuming that Assumption \ref{thm ass1} holds
for $\beta=\alpha$ with constant $c$ and
$|||(y^0,\dot{y}^0)|||_{s}\leq M,$ one has
$|||(y^1,\dot{y}^1)|||_{s}  \leq C$ with a constant $C$ depending
only on $M, s, p,$ and $c$.
\end{lem}

The proof is quite similar to that of Lemma \ref{lem-regularity} and
thus we ignore it.

\begin{lem}\label{lem-local error-2}
 (Local error in $H^{s+1-\alpha}\times H^{s-\alpha}$ for $0<\alpha\leq 1$.)
Under the conditions of Lemma \ref{lem-regularity-new}, if
$$|||(y(\tau),\dot{y}(\tau))|||_{s}\leq M$$ for $t_0\leq \tau \leq
t_1,$ then, it holds that
$$|||(y(t_1),\dot{y}(t_1))-(y^1,\dot{y}^1)|||_{s-\alpha}\leq Ch^{2+\alpha},$$
where the constant $C$ depends only on $M, s, p,$ and $c$.
\end{lem}
\begin{proof}  In a similar way to  the proof of Lemma \ref{lem-local error},
$C$ is denoted as a generic constant which  depends on $M, s, p,$
and $c$.

\vskip2mm \textbf{(I) Local error of $y(t_1)-y^1$.} According to
  \eqref{vari cons}, \eqref{methods} and the fact that
$$\int_{t_0}^{t_1}(t_1-\tau)\textmd{sinc}((t_1-\tau)\Omega) d\tau=\frac{1}{2}h^2\textmd{sinc}^2\big(\frac{1}{2}h\Omega\big),$$
 we obtain
\begin{eqnarray}
  y(t_1)-y^1&= &
\displaystyle\int_{t_0}^{t_1}(t_1-\tau)\textmd{sinc}((t_1-\tau)\Omega)f(y(\tau))d
\tau-h^{2}
 \bar{b}_{1}(V)f(y^{ \frac{1}{2}})   \nonumber \\
   &=& \displaystyle\int_{t_0}^{t_1}(t_1-\tau)\textmd{sinc}((t_1-\tau)\Omega)\big[ f(y(\tau))-f\big(y
   \big(\frac{t_0+t_1}{2}\big)\big)\big]d
\tau   \label{er pr 1-new}\\
 &+ &\frac{1}{2}h^2\textmd{sinc}^2\big(\frac{1}{2}h\Omega\big)
\big[f\big(y\big(\frac{t_0+t_1}{2}\big)\big)-f(y^{\frac{1}{2}})  \big]  \label{er pr 2-new}\\
&+ & h^2 \big[\frac{1}{2}\textmd{sinc}^2\big(\frac{1}{2}h\Omega\big)
-\bar{b}_{1}(V)\big]f(y^{ \frac{1}{2}}).  \label{er pr 3-new}
\end{eqnarray}

  $\bullet$  Bound of  \eqref{er pr 1-new}.
For $\xi>0$ and $0<\alpha\leq 1$, we have
 $|\textmd{sinc}(\xi)|  \leq
 \xi^{-1+\alpha}.$ By this result,   the estimate \eqref{norm pro2-1}
with $\sigma=s$, and the fact $$\norm{y(\tau)- y
   \big(\frac{t_0+t_1}{2}\big)}_{s}  \leq\int_{\frac{t_0+t_1}{2}}^{\tau} \norm{\dot{y}(t)}_{s} d
t\leq Ch,$$ we have
\begin{equation*}
\begin{array}[c]{ll}&\norm{\displaystyle\int_{t_0}^{t_1}(t_1-\tau)\textmd{sinc}((t_1-\tau)\Omega)\big[ f(y(\tau))-f\big(y
   \big(\frac{t_0+t_1}{2}\big)\big)\big]d
\tau }_{s+1-\alpha} \\
&\leq  h^{-1+\alpha}\displaystyle\int_{t_0}^{t_1} |t_1-\tau| \norm{
f(y(\tau))-f\big(y\big(\frac{t_0+t_1}{2}\big)\big)}_{s}d \tau \\
&\leq C h^{-1+\alpha}\displaystyle\int_{t_0}^{t_1} |t_1-\tau| \norm{
 y(\tau) -y\big(\frac{t_0+t_1}{2}\big)}_{s}d \tau \\
&\leq C h^{\alpha} \displaystyle\int_{t_0}^{t_1}|t_1-\tau|   d \tau
\leq C h^{2+\alpha}.\end{array}
\end{equation*}

 $\bullet$   For \eqref{er pr 2-new}, on basis of  the fact that $|\textmd{sinc}(\xi)|^2 \leq \frac{1\cdot \xi}{\xi^2}
= \xi^{-1}$ for $\xi>0$ and the estimate \eqref{norm pro2-1} with
$\sigma=s-\alpha$, one arrives at
\begin{equation*}
\begin{array}[c]{ll}&\norm{\frac{1}{2}h^2\textmd{sinc}^2\big(\frac{1}{2}h\Omega\big)
\big[f\big(y\big(\frac{t_0+t_1}{2}\big)\big)-f(y^{\frac{1}{2}})
\big] }_{s+1-\alpha} \\
\leq&   C   h \norm{
 f\big(y\big(\frac{t_0+t_1}{2}\big)\big)-f(y^{\frac{1}{2}})
}_{s-\alpha} \leq  C   h \norm{y\big(\frac{t_0+t_1}{2}\big) -
y^{\frac{1}{2}} }_{s-\alpha}.
\end{array}
\end{equation*}
On the other hand, by the   the estimate \eqref{norm pro2-2} with
$\sigma'=s+1$, it is  yielded  that
\begin{equation}\label{err-y1-new}
\begin{array}[c]{ll}
&\norm{y\big(\frac{t_0+t_1}{2}\big) - y^{\frac{1}{2}}
}_{s-\alpha}\leq
\norm{y(\frac{t_0+t_1}{2})-y^{\frac{1}{2}}}_{s+2-\alpha}\\
=&\norm{\displaystyle\int_{t_0}^{\frac{t_0+t_1}{2}}(\frac{t_0+t_1}{2}-\tau)\textmd{sinc}((\frac{t_0+t_1}{2}-\tau)\Omega)f(y(\tau))d
\tau}_{s+2-\alpha}\\
\leq&h^{-1+\alpha} \displaystyle\int_{t_0}^{\frac{t_0+t_1}{2}}|
\frac{t_0+t_1}{2}-\tau| \norm{f(y(\tau))}_{s+1} d \tau\\
\leq& h^{-1+\alpha} \displaystyle\int_{t_0}^{\frac{t_0+t_1}{2}}|
\frac{t_0+t_1}{2}-\tau| C d \tau\leq C h^{1+\alpha}.
\end{array}
\end{equation}
Thence
\begin{equation*}
\begin{array}[c]{ll}&\norm{\frac{1}{2}h^2\textmd{sinc}^2\big(\frac{1}{2}h\Omega\big)
\big[f\big(y\big(\frac{t_0+t_1}{2}\big)\big)-f(y^{\frac{1}{2}})
\big] }_{s+1-\alpha} \leq   C h^{2+\alpha}.
\end{array}
\end{equation*}

  $\bullet$  For \eqref{er pr 3-new}, in the light of \eqref{ass2} and the estimate
\eqref{norm pro2-2} with  $\sigma'=s+1$, we have
\begin{equation*}
\begin{array}[c]{ll}&\norm{h^2 \big[\frac{1}{2}\textmd{sinc}^2\big(\frac{1}{2}h\Omega\big)
-\bar{b}_{1}(V)\big]f(y^{ \frac{1}{2}}) }_{s+1-\alpha}\leq
h^{2+\alpha}\norm{f(y^{ \frac{1}{2}}) }_{s+1} \leq   C h^{2+\alpha}.
\end{array}
\end{equation*}

 From  the above bounds of   \eqref{er pr 1-new}-\eqref{er pr 3-new},
it follows that $\norm{y(t_1)-y^1 }_{s+1-\alpha} \leq   C
h^{2+\alpha}.$

\vskip2mm \textbf{(II) Local error of $\dot{y}(t_1)-\dot{y}^1$.}
 The expression of $\dot{y}(t_1)-\dot{y}^1$ is
\begin{eqnarray}
   &\ & \dot{y}(t_1)- \dot{y}^1  \nonumber \\
   &= & \displaystyle\int_{t_0}^{t_1}[\cos((t_1-\tau)\Omega)-I_d]f(y(\tau))d \tau  \label{dy1 er 1-new}\\
    &+ & \displaystyle\int_{t_0}^{t_1} f(y(\tau))d \tau-h
  f\big(y \big(\frac{t_0+t_1}{2}\big)\big) \label{dy1 er 2-new}\\
    &+ & h
  f\big(y \big(\frac{t_0+t_1}{2}\big)\big)  -h
 f(y^{ \frac{1}{2}})\label{dy1 er 3-new}\\
     &+ & h
  (I_d- b_{1}(V))
f(y^{ \frac{1}{2}}).\label{dy1 er 4-new}
\end{eqnarray}

   $\bullet$  Bound of \eqref{dy1 er 1-new} is obtained in the same way as the
bound of \eqref{dy1 er 1}.

  $\bullet$  Bound of \eqref{dy1 er 2-new}. Since $1 \leq \xi^{ \alpha} + \xi^{-1+ \alpha}$ for  $\xi >
0$, \eqref{dy1 er 2-new} can be formulated as
\begin{equation*}
\begin{array}[c]{ll}
&\norm{\displaystyle\int_{t_0}^{t_1} f(y(\tau))d \tau-h
  f\big(y \big(\frac{t_0+t_1}{2}\big)\big)}_{s-\alpha} \\
  \leq& h^{ \alpha}\norm{\displaystyle\int_{t_0}^{t_1} f(y(\tau))d \tau-h
  f\big(y \big(\frac{t_0+t_1}{2}\big)\big)}_{s}  +  h^{-1+\alpha}\norm{\displaystyle\int_{t_0}^{t_1} f(y(\tau))d \tau-h
  f\big(y \big(\frac{t_0+t_1}{2}\big)\big)}_{s-1}.
\end{array}
\end{equation*}
It follows from \eqref{err-f01} that
$\norm{\displaystyle\int_{t_0}^{t_1} f(y(\tau))d \tau-h
  f\big(y \big(\frac{t_0+t_1}{2}\big)\big)}_{s}   \leq Ch^2.$
For an estimate in the norm $\norm{\cdot}_{s-1}$, by the Peano
kernel theorem and \eqref{norm pro3-2},  we   obtain
$$\norm{\displaystyle\int_{t_0}^{t_1} f(y(\tau))d \tau-h
  f\big(y \big(\frac{t_0+t_1}{2}\big)\big)}_{s-1}
  =h^3\norm{\displaystyle\int_{t_0}^{t_1} K_2(\sigma) \frac{d^2}{dt^2} f(y(t_0+\sigma h))d \tau }_{s-1} \leq
  Ch^3,$$
where $K_2(\sigma)$ is the second-order Peano kernel of the
mid-point rule. Thus, one gets
$$\norm{\displaystyle\int_{t_0}^{t_1} f(y(\tau))d \tau-h
  f\big(y \big(\frac{t_0+t_1}{2}\big)\big)}_{s-\alpha}\leq C h^{2+\alpha}.$$

 $\bullet$   Bound of \eqref{dy1 er 3-new}.
By \eqref{err-y1-new} and \eqref{norm pro2-1} with
$\sigma=s-\alpha$, one has
$$\norm{h
  f\big(y \big(\frac{t_0+t_1}{2}\big)\big)  -h
 f(y^{ \frac{1}{2}}) }_{s-\alpha}\leq C h\norm{
   y \big(\frac{t_0+t_1}{2}\big)- y^{ \frac{1}{2}} }_{s-\alpha}\leq C h^{2+\alpha}.$$

  $\bullet$  Bound of \eqref{dy1 er 4-new}.  From  \eqref{fhaf1 bound} and
\eqref{ass3}, it follows immediately that
$$\norm{h (I_d- b_{1}(V))
f(y^{ \frac{1}{2}})}_{s-\alpha}  \leq C h^{2+\alpha} \norm{  f(y^{
\frac{1}{2}})}_{s+1}   \leq C h^{2+\alpha}.$$

In summary, we have $\norm{\dot{y}(t_1)-\dot{y}^1 }_{s-\alpha} \leq
C h^{2+\alpha} $ and the proof is complete.
\end{proof}

\begin{lem}\label{lem-stabi-new}
(Stability   in $H^{s+1-\alpha}\times H^{s-\alpha}$ for
$0<\alpha\leq 1$.) Under the conditions of Lemma
\ref{lem-regularity-new},  consider
 different initial values
$(y_0,\dot{y}_0)$ and $(z_0,\dot{z}_0)$ for the ERKN integrator
\eqref{methods}. If
$$|||(y_0,\dot{y}_0)|||_{s}\leq M\qquad \textmd{and} \qquad |||(z_0,\dot{z}_0)|||_{s}\leq M,$$
then we have
$$|||(y^1,\dot{y}^1)-(z^1,\dot{z}^1)|||_{s-\alpha}\leq (1+Ch) |||(y^0,\dot{y}^0)-(z^0,\dot{z}^0)|||_{s-\alpha},$$
where the constant $C$ depends only on $M, s, p,$ and $c$.
\end{lem}

\begin{proof}
Consider
\begin{eqnarray}
  |||(y^1,\dot{y}^1)-(z^1,\dot{z}^1)|||_{s-\alpha} &\leq&
  |||(y^0,\dot{y}^0)-(z^0,\dot{z}^0)|||_{s-\alpha}
    +   h |\dot{y}^0_0-\dot{z}^0_0| \nonumber   \\
    &+ & h^2\norm{\bar{b}_1(V)  \big(f(y^{\frac{1}{2}}) -f(z^{\frac{1}{2}})\big) }_{s+1-\alpha} \label{sta er 2-new}\\
    &+ & h \norm{b_1(V)  \big(f(y^{\frac{1}{2}}) -f(z^{\frac{1}{2}})\big) }_{s-\alpha}. \label{sta er 3-new}
\end{eqnarray}
It follows from \eqref{ass2} that
$|\bar{b}_1(\xi)|\leq\frac{1}{2}+c\xi^{\alpha}$. Then the bound of
\eqref{sta er 2-new} is
\begin{equation*}
\begin{array}[c]{ll}&h^2\norm{\bar{b}_1(V)  \big(f(y^{\frac{1}{2}}) -f(z^{\frac{1}{2}})\big)
}_{s+1-\alpha} \leq \frac{1}{2}h^2\norm{ f(y^{\frac{1}{2}})
-f(z^{\frac{1}{2}}) }_{s+1-\alpha}\\
& +\frac{1}{2}h^{2+\alpha}\norm{ f(y^{\frac{1}{2}})
-f(z^{\frac{1}{2}}) }_{s+1}\leq C h^{2} (1/2+h^{\alpha})\norm{
y^{\frac{1}{2}}-z^{\frac{1}{2}} }_{s+1}\\
&\leq C h^{2 }(1/2+h^{\alpha})
 \norm{ y^0-z^0}_{s+1}+ C h^{2 }(1/2+h^{\alpha})
\norm{\dot{y}^0-\dot{z}^0}_{s}.
\end{array}
\end{equation*}
With regard to \eqref{sta er 3-new}, since
$|b_1(\xi)|\leq1+c\xi^{1+\alpha}$ from \eqref{ass3}, we obtain
\begin{equation*}
\begin{array}[c]{ll}&h \norm{b_1(V)   f(y^{\frac{1}{2}}) -f(z^{\frac{1}{2}})  }_{s-\alpha}
\leq   h  \norm{  f(y^{\frac{1}{2}}) -f(z^{\frac{1}{2}})
}_{s-\alpha} +ch^{2+\alpha} \norm{ f(y^{\frac{1}{2}})
-f(z^{\frac{1}{2}})  }_{s+1} \\
&\leq  C (h+h^{2+\alpha})  \norm{  y^0-z^0}_{s+1} +  C
(h+h^{2+\alpha}) \norm{\dot{y}^0-\dot{z}^0}_{s}.
\end{array}
\end{equation*}
The proof is complete by considering the above bounds.
\end{proof}

\vskip4mm
 \textbf{Proof  of Theorem \ref{thm main-re} for $0<\alpha\leq1$.}
\begin{proof} It is noted that this proof is   same to that for $-1\leq\alpha<0$
given in Section \ref{sec-prof1}. One important fact  used  here is
that the numerical solution is bounded in $H^{s+1}\times H^s$, which
is obtained by Remark \ref{re regularity num}. \end{proof}

\section{Error analysis for the spatial semidiscretization by finite differences}
\label{sec:finite differences} For the spatial discretization by
finite differences, here we approximate the derivative $u_{xx}(x,
t)$  by using second-order symmetric differences
$$\frac{u(x+\Delta x, t)-2u(x, t)+u(x-\Delta x, t)}{(\Delta x)^2}\quad \textmd{with} \quad \Delta x=\frac{\pi}{K}.$$
Then inserting  the points $x_k =  \frac{k\pi}{K}$ in the equation
and defining the vector $y = (y_j)_{j\in \mathcal{K}}$ by $u(x_k, t)
= \sum\limits_{j\in\mathcal{K}}
 y_j(t)\mathrm{e}^{\mathrm{i}jx}$ for $k\in \mathcal{K}$, we can  again obtain a system of ODEs of the form \eqref{prob}.
 This system has the same nonlinearity as in \eqref{prob} but different
 frequencies. The new frequencies  are
$$\omega_j=\frac{2}{\Delta x}\mid \sin\big(\frac{j\Delta x}{2}\big)\mid, \quad j\in \mathcal{K},$$
which satisfy $\frac{2}{\pi}|j|\leq \omega_j \leq  |j|. $

Define a new norm
\begin{equation*}
\norm{y}^{new}_s:=\Big(\sum\limits_{ j \in\mathcal{K}}
\max(\omega_j,\omega_{\min})^{2s}|y_j|^2 \Big)^{1/2},
\label{norm-new}%
\end{equation*}
  where $\omega_{\min}$ denotes the minimal nonzero frequency.
The proof of Theorem \ref{thm main-re} can be  transferred  with the
new norm  to the new system and thus we have the following error
bound of ERKN integrators.

\begin{theo}\label{thm main-re-new}
Assume that
\begin{itemize}
\item   $c\geq 1$ and $s\geq 0$;

\item  the exact solution
$(y(t), \dot{y}(t))$ of the spatial semidiscretization by finite
differences
   satisfies
\begin{equation*}\norm{y(t)}^{new}_{s+1}+ \norm{\dot{y}(t)}^{new}_{s}\leq M \qquad \textmd{for} \qquad 0\leq
t-t_0 \leq T;
\end{equation*}

\item Assumption \ref{thm ass1} is true with constant $c$ for $\beta= 0$
and $\beta= \alpha$ with   $-1 \leq\alpha\leq 1$.

\end{itemize}
Then, there exists $h_0 > 0$ such that for  $0<h \leq h_0$, the
error bound of ERKN integrators for the spatial semidiscretization
by finite differences  is $$\norm{y(t_n)-y^n}^{new}_{s+1-\alpha}+
\norm{\dot{y}(t_n)-\dot{y}^n}^{new}_{s-\alpha}\leq Ch^{1+\alpha}
\qquad \textmd{for} \qquad 0\leq t_n-t_0=nh \leq T,$$  where the
constants $C$ and $h_0$ depend   on $M,s,p,T,$ and $c$.
\end{theo}

\section{Numerical experiments}
\label{sec:experiments} In this section, we carry out a numerical
experiment to illustrate the error bounds of  two one-stage explicit
ERKN integrators.

\begin{figure}[ptb]
\centering\tabcolsep=2mm
\begin{tabular}
[l]{lll}%
\includegraphics[width=6cm,height=3cm]{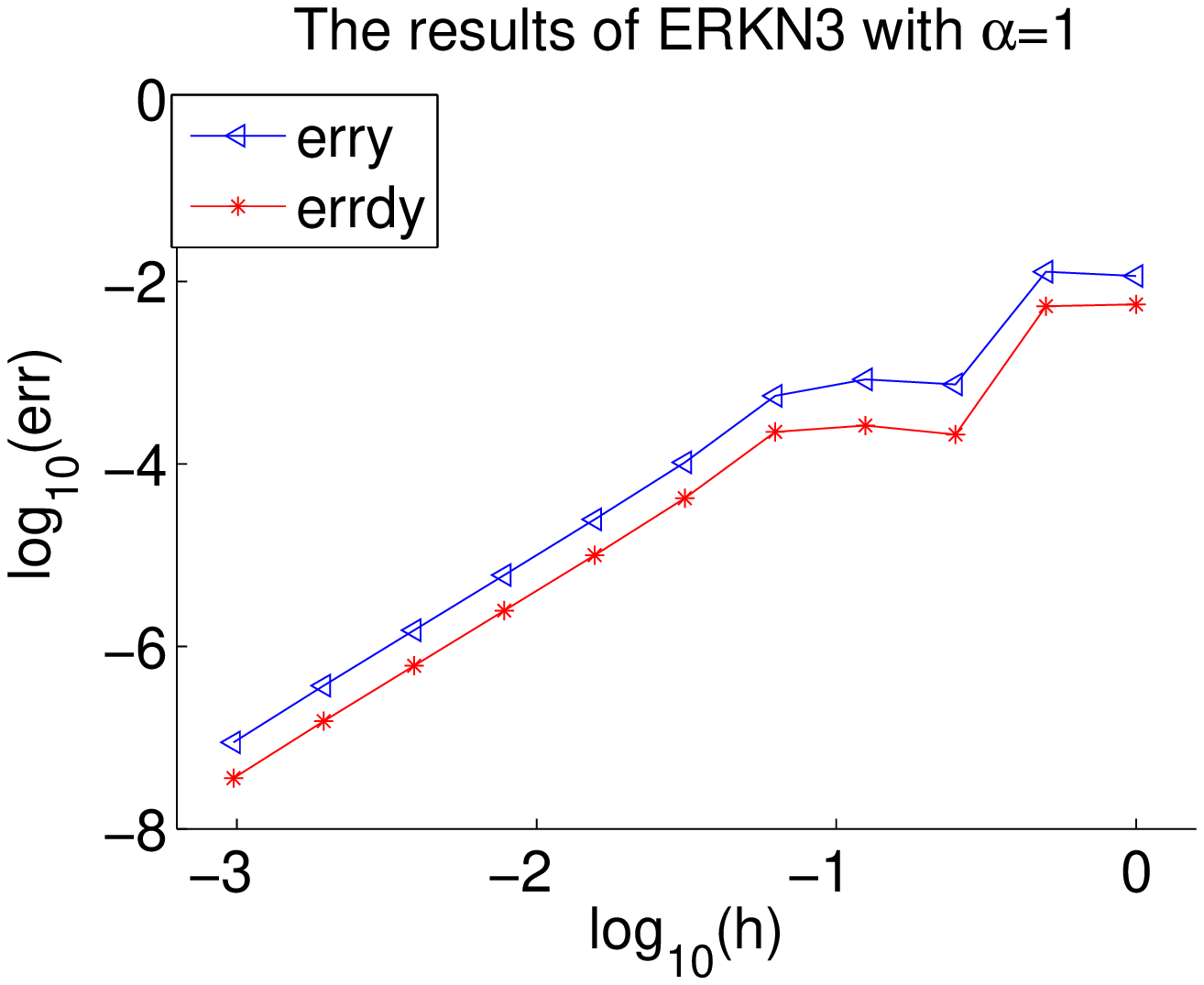}
\includegraphics[width=6cm,height=3cm]{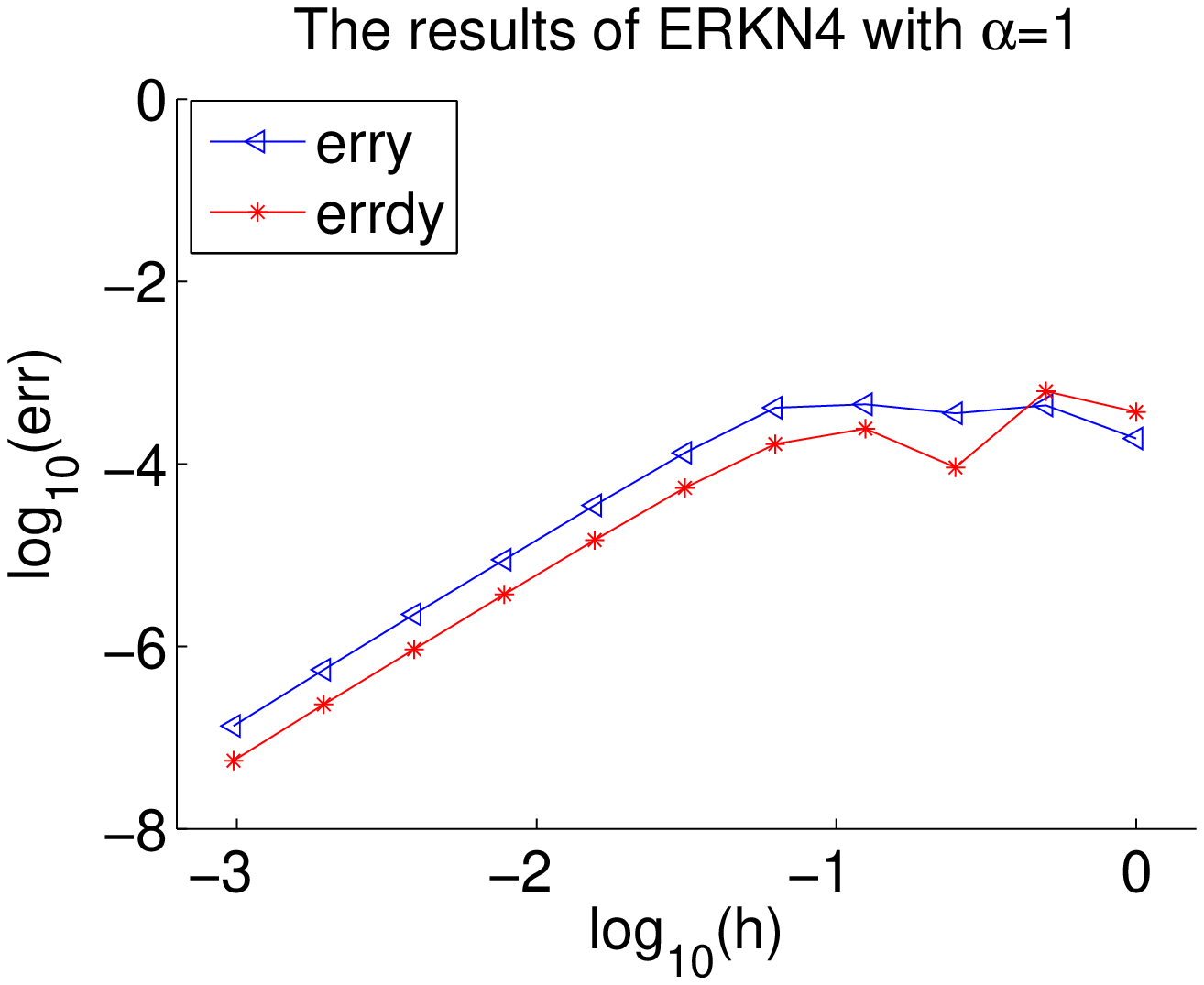}  \\
\includegraphics[width=6cm,height=3cm]{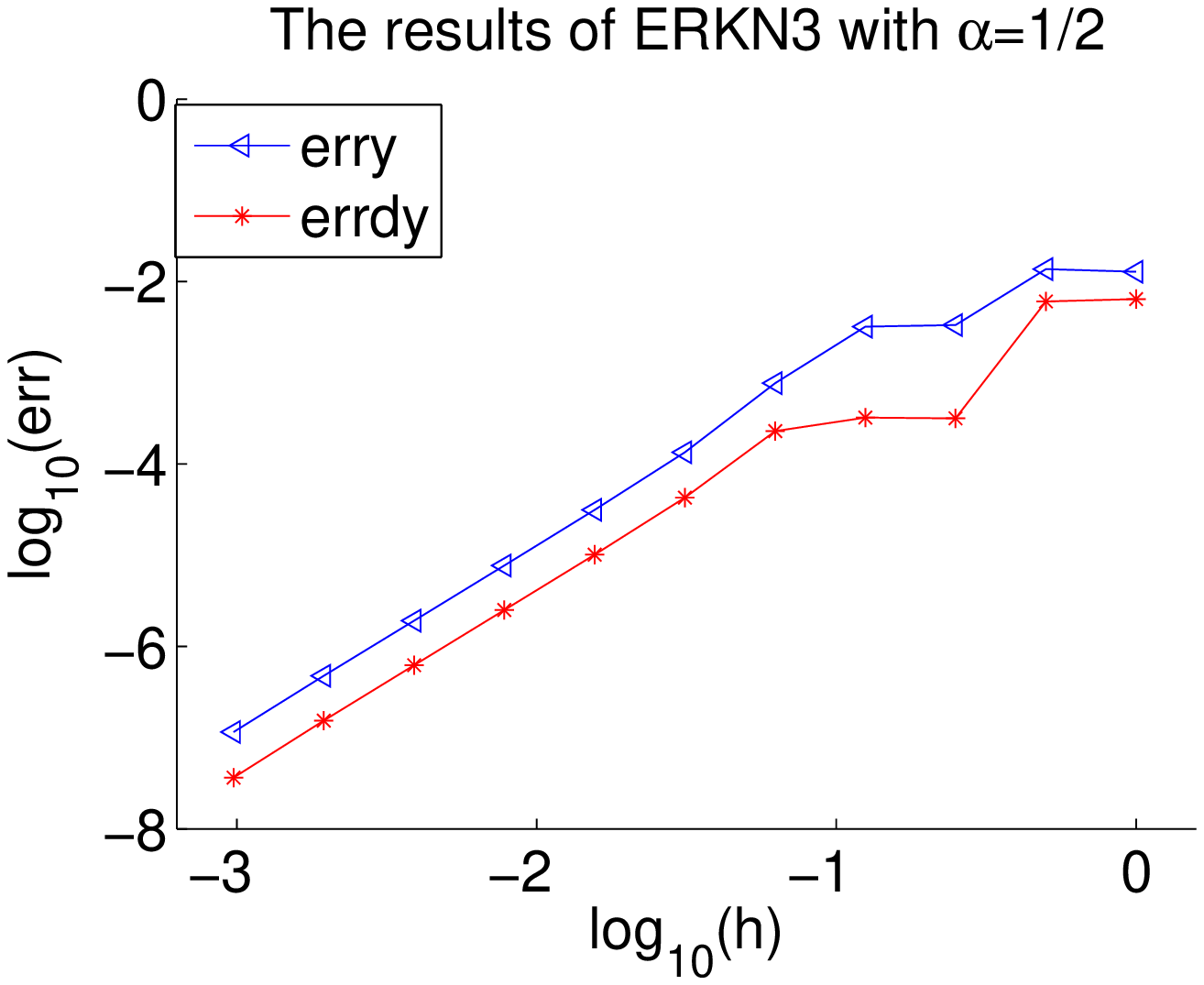}
\includegraphics[width=6cm,height=3cm]{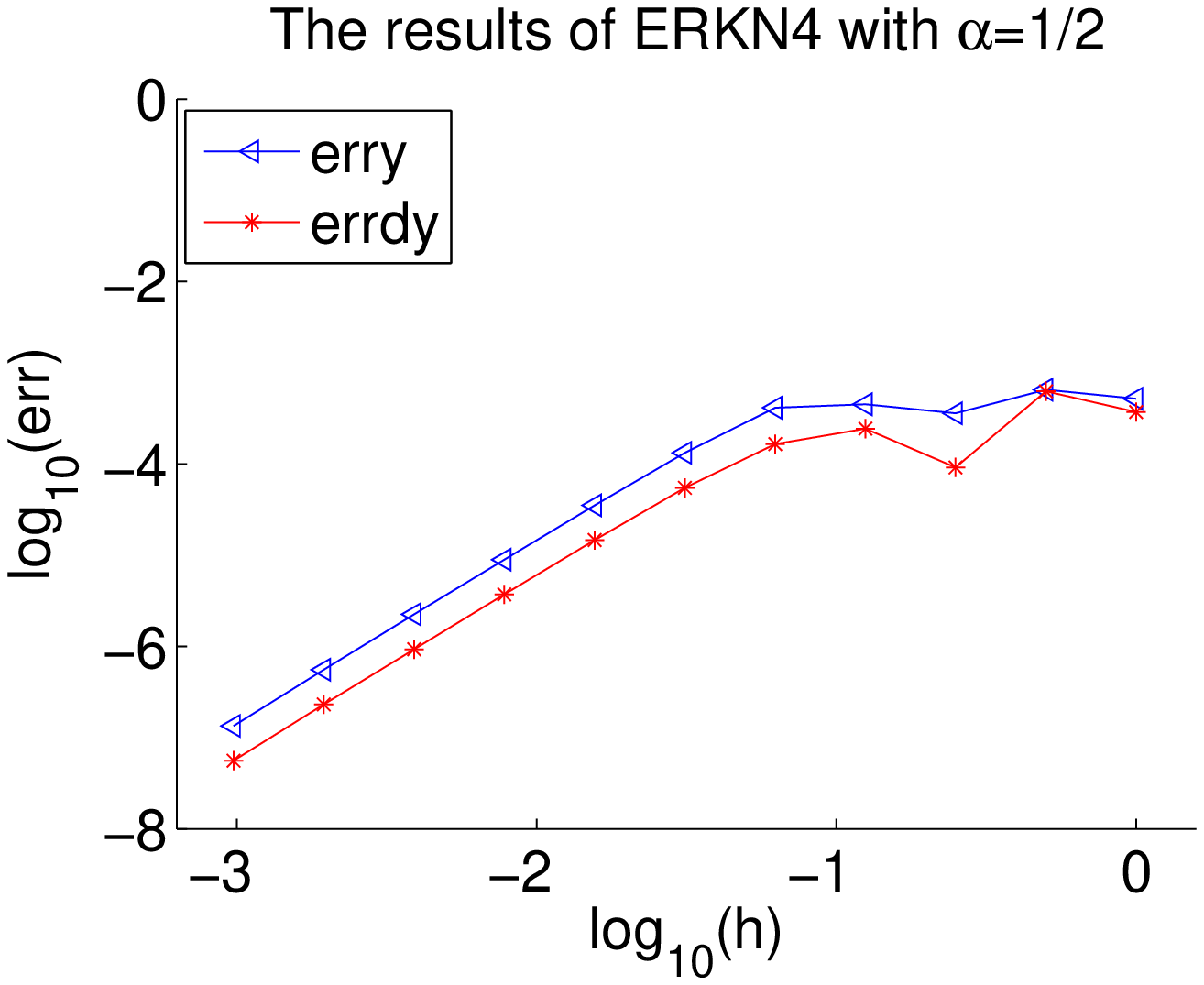}  \\
\includegraphics[width=6cm,height=3cm]{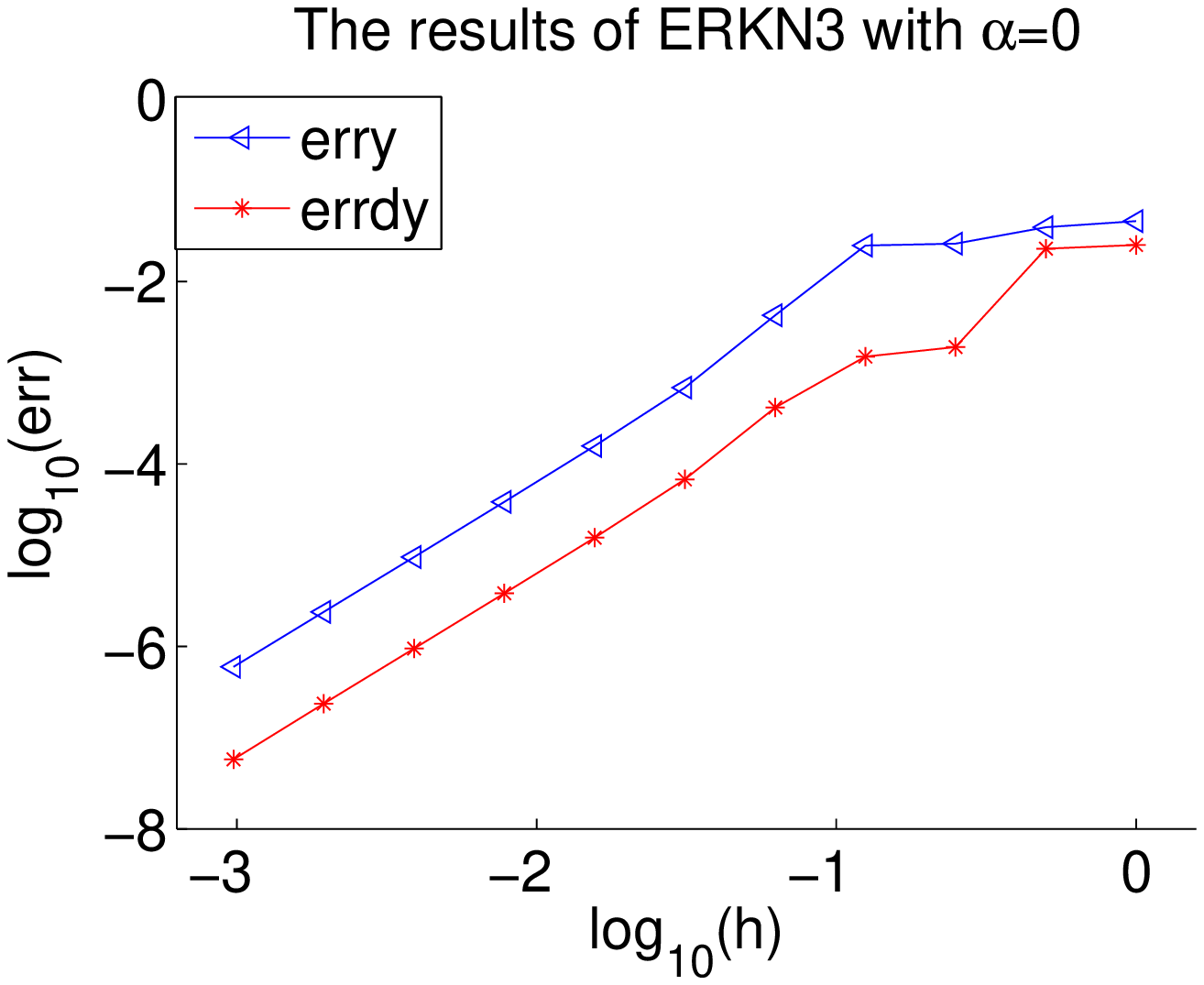}
\includegraphics[width=6cm,height=3cm]{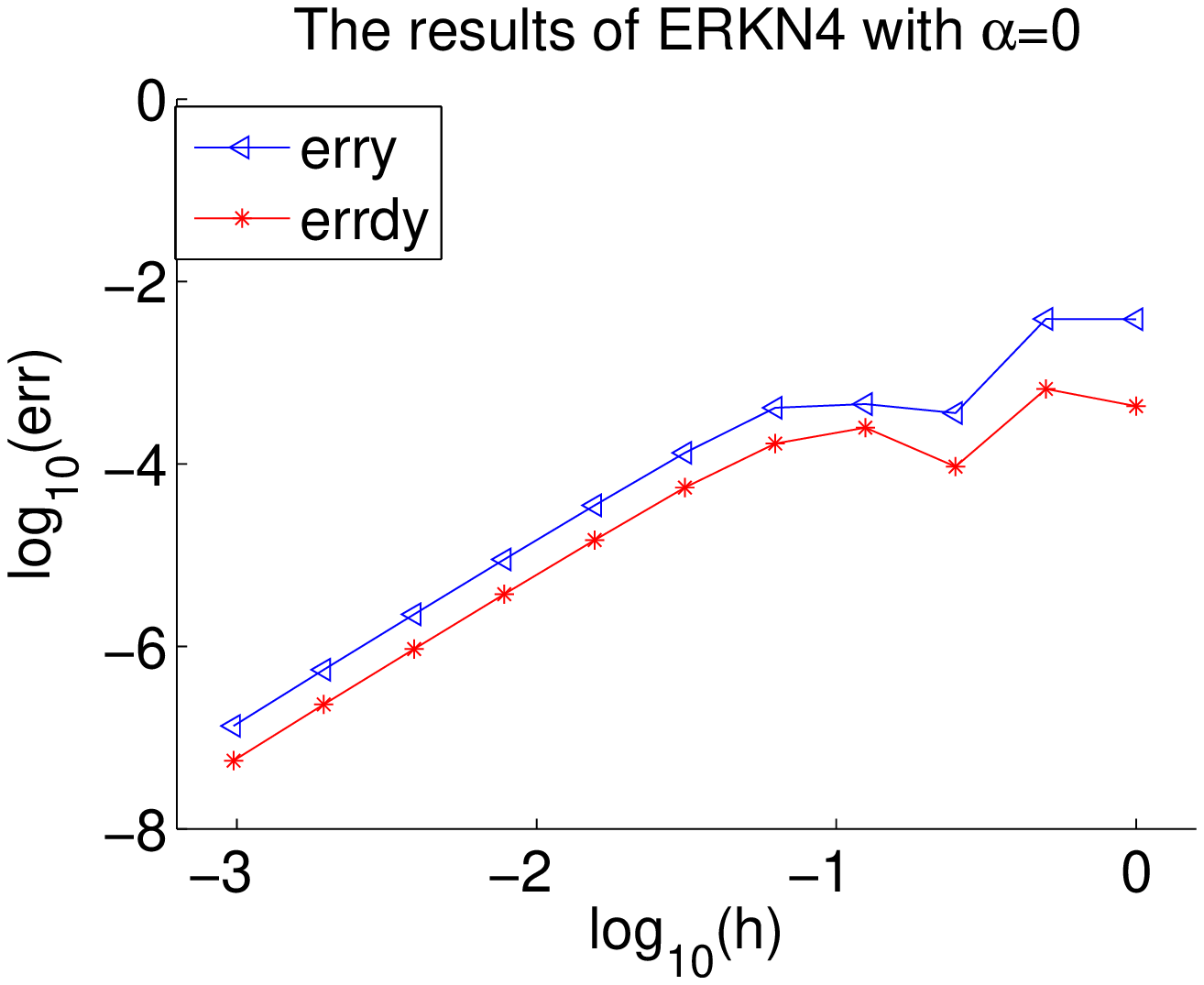}  \\
\includegraphics[width=6cm,height=3cm]{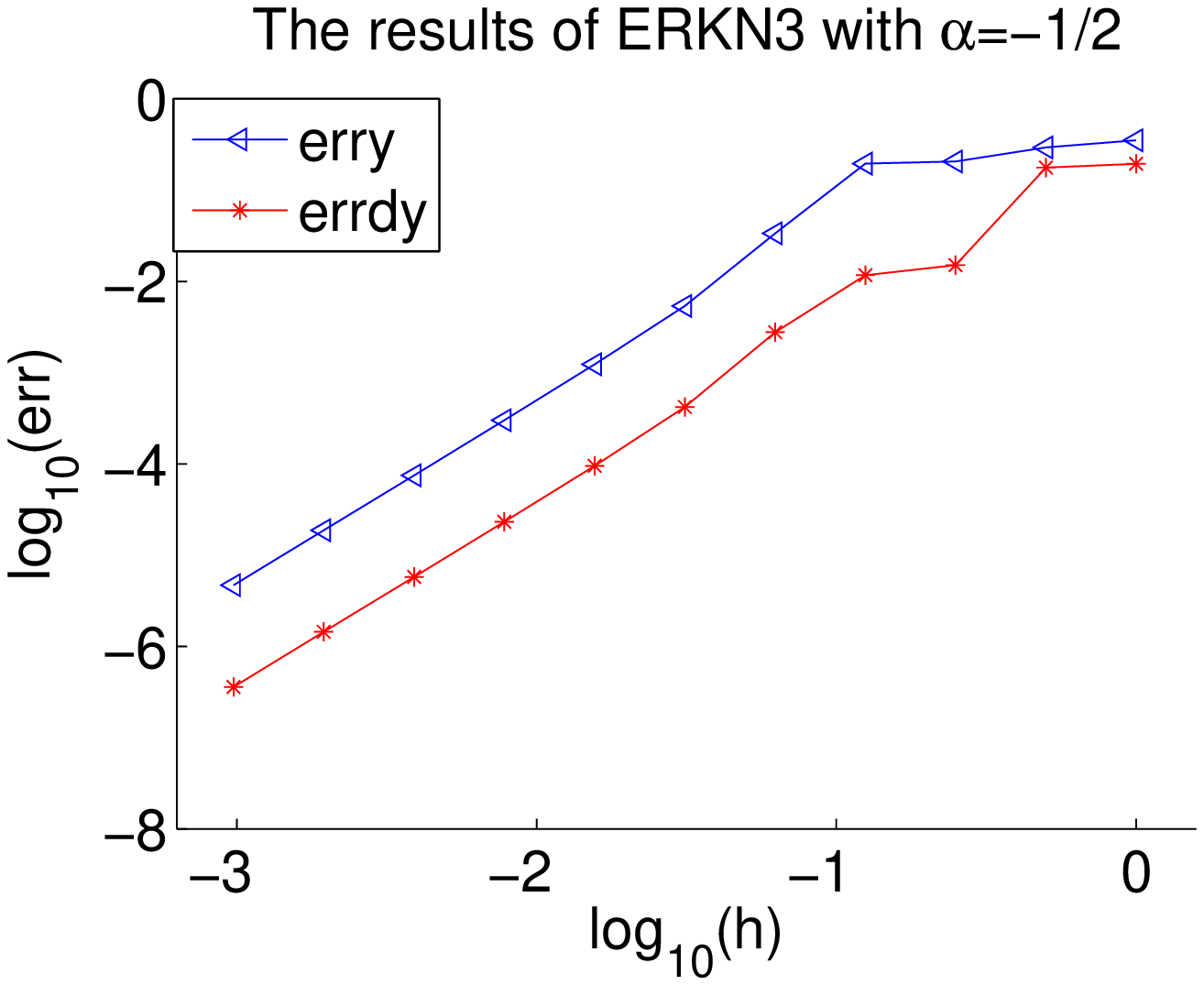}
\includegraphics[width=6cm,height=3cm]{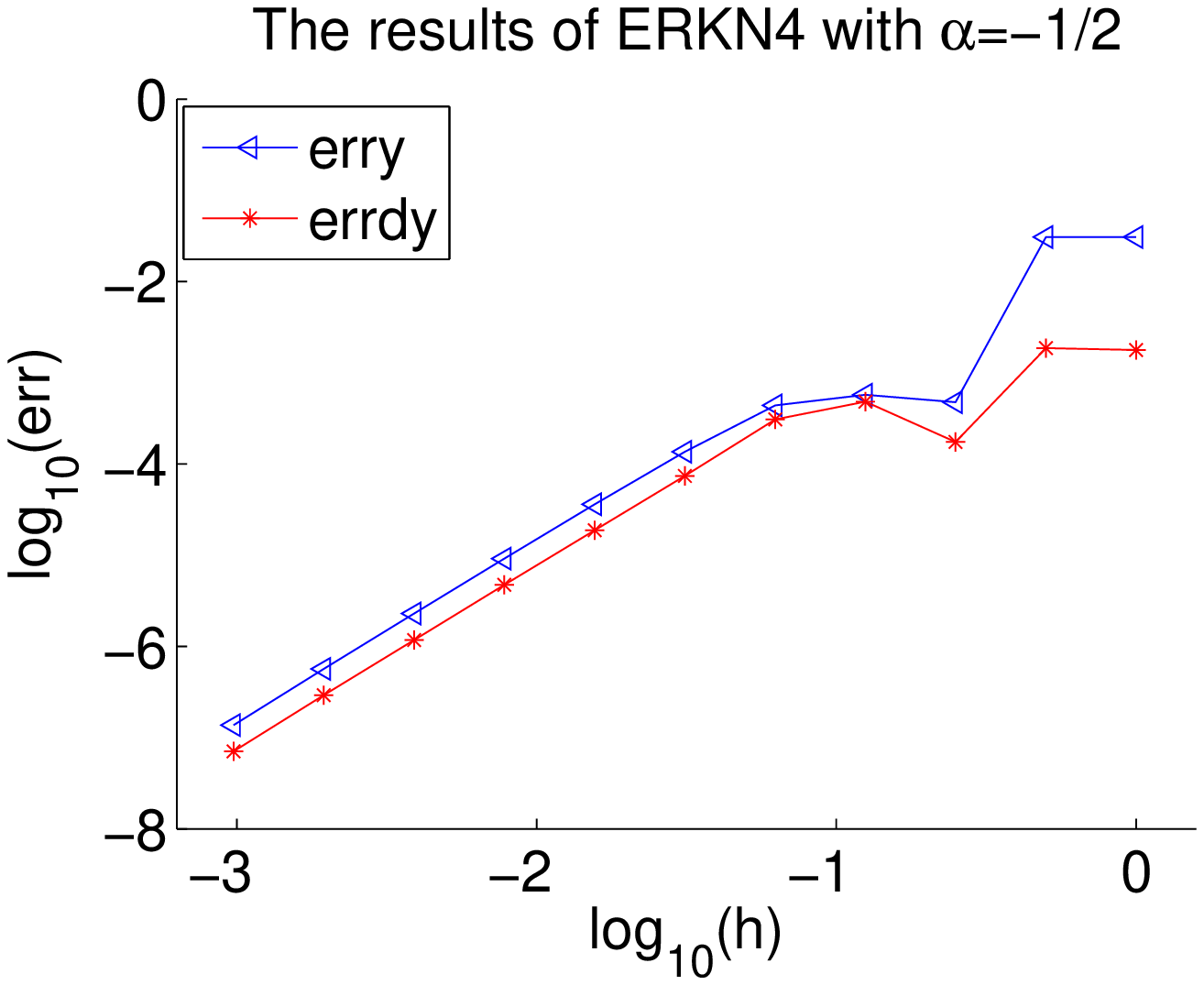}  \\
\includegraphics[width=6cm,height=3cm]{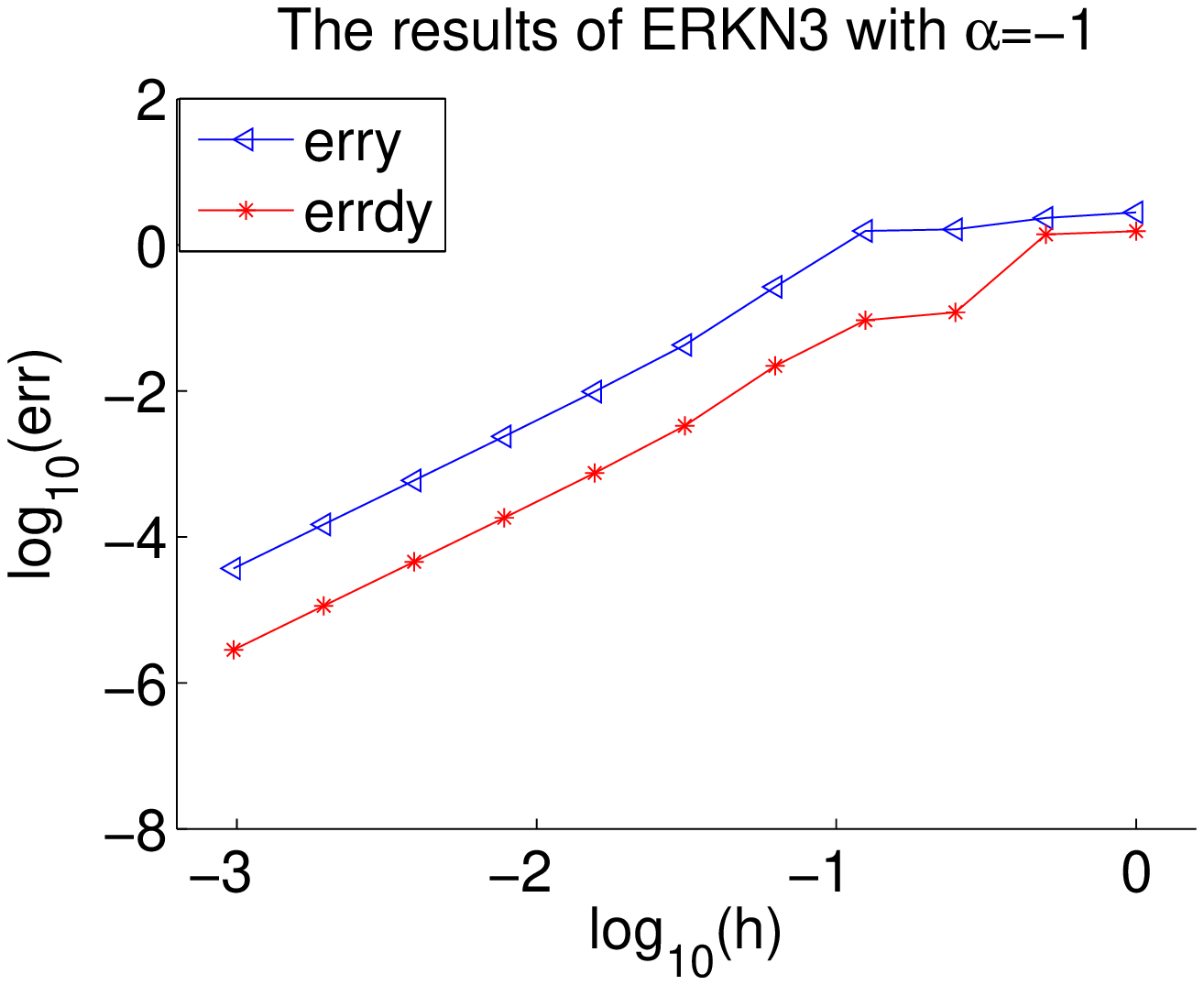}
\includegraphics[width=6cm,height=3cm]{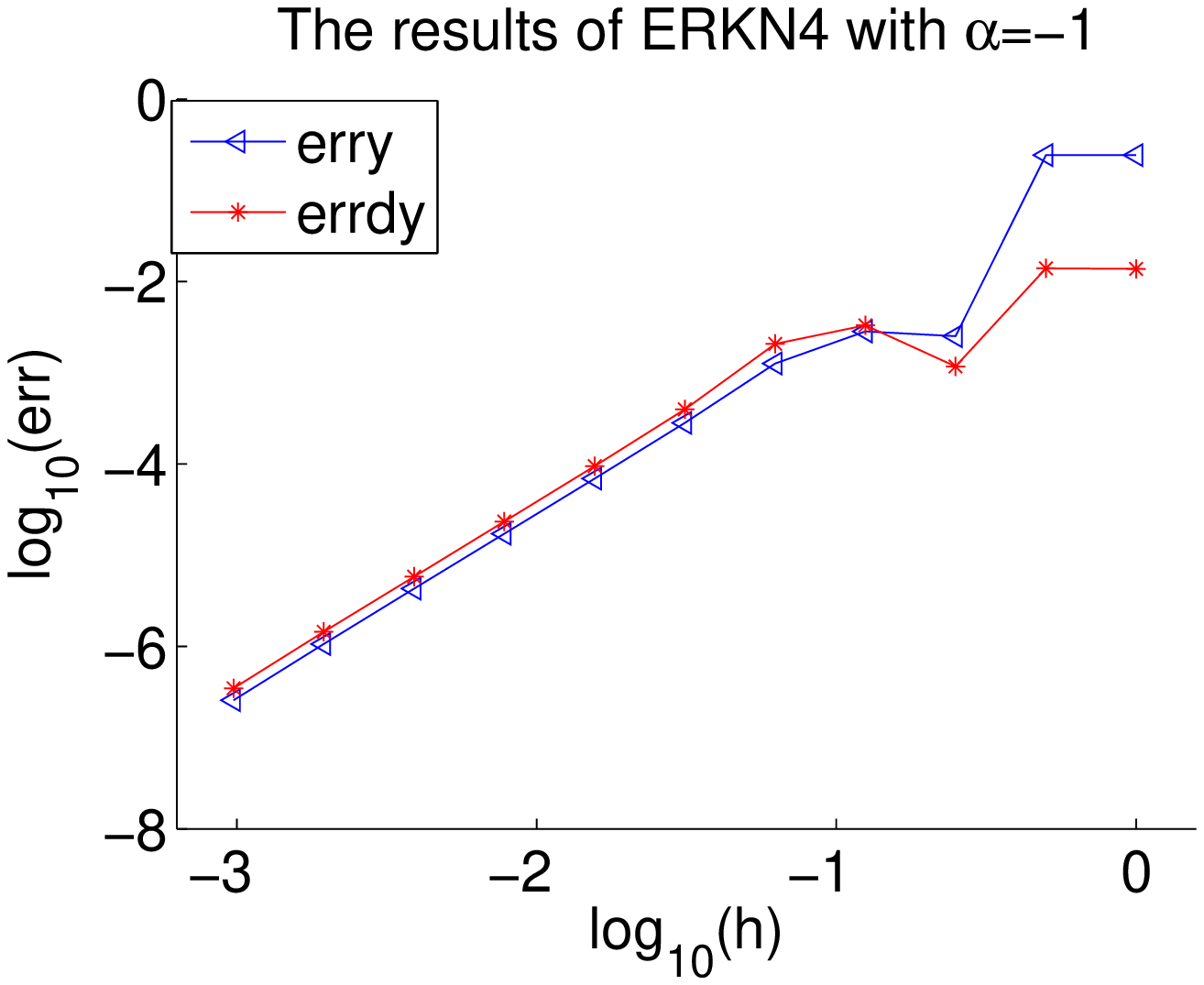}
\end{tabular}
\caption{The logarithm of the  errors    against the logarithm of stepsizes for $K=2^6$.}%
\label{fig0}%
\end{figure}
We consider $p=2$ for the problem \eqref{wave equa} and use the
spatial discretization with $K=2^6,2^8$. Following
\cite{Gauckler15}, we choose the initial conditions for the
coefficients $y_j(t_0)$ and $ \dot{y}_j(t_0)$ on the complex unit
circle and then scale them by $\langle j\rangle^{-1.51}$ and
$\langle j\rangle^{-0.51}$, respectively.  It is noted that these
complex numbers are chosen such that  the corresponding
trigonometric polynomial \eqref{trigo pol} takes real values in the
collocation points. In this way, the corresponding initial values
satisfy the condition \eqref{thm con1} of Theorem \ref{thm main-re}
at time $t = t_0$ uniformly in $K$ for $s = 0$. For the
discretization in time, we choose ERKN3 and ERKN 4 whose
coefficients are displayed  in Table \ref{praERKN}. The problem is
solved  in the interval $[0,10]$ with stepsizes $h=1/2^{j}$ for
$j=0,1,\ldots,10.$ We measure the errors
$$\textmd{erry}=\norm{y(t_n)-y^n}_{1-\alpha},\ \ \textmd{errdy}=\norm{\dot{y}(t_n)-\dot{y}^n}_{-\alpha}$$
in different Sobolev norms $\alpha=1,\frac{1}{2},0,-\frac{1}{2},-1.$
We plot the logarithm of the  errors    against the logarithm of
stepsizes for these two integrators and see Figures \ref{fig0} and
\ref{fig1} for the results of $K=2^6$ and $K=2^8$, respectively.
From these results, it follows that the convergent order is not
uniform for $\alpha$ (especially for the case $K=2^8$) and the
observed order of convergence is about $1 + \alpha$, which soundly
supports the result given in Theorem \ref{thm main-re}.

\begin{figure}[ptb]
\centering\tabcolsep=2mm
\begin{tabular}
[l]{lll}%
\includegraphics[width=6cm,height=3cm]{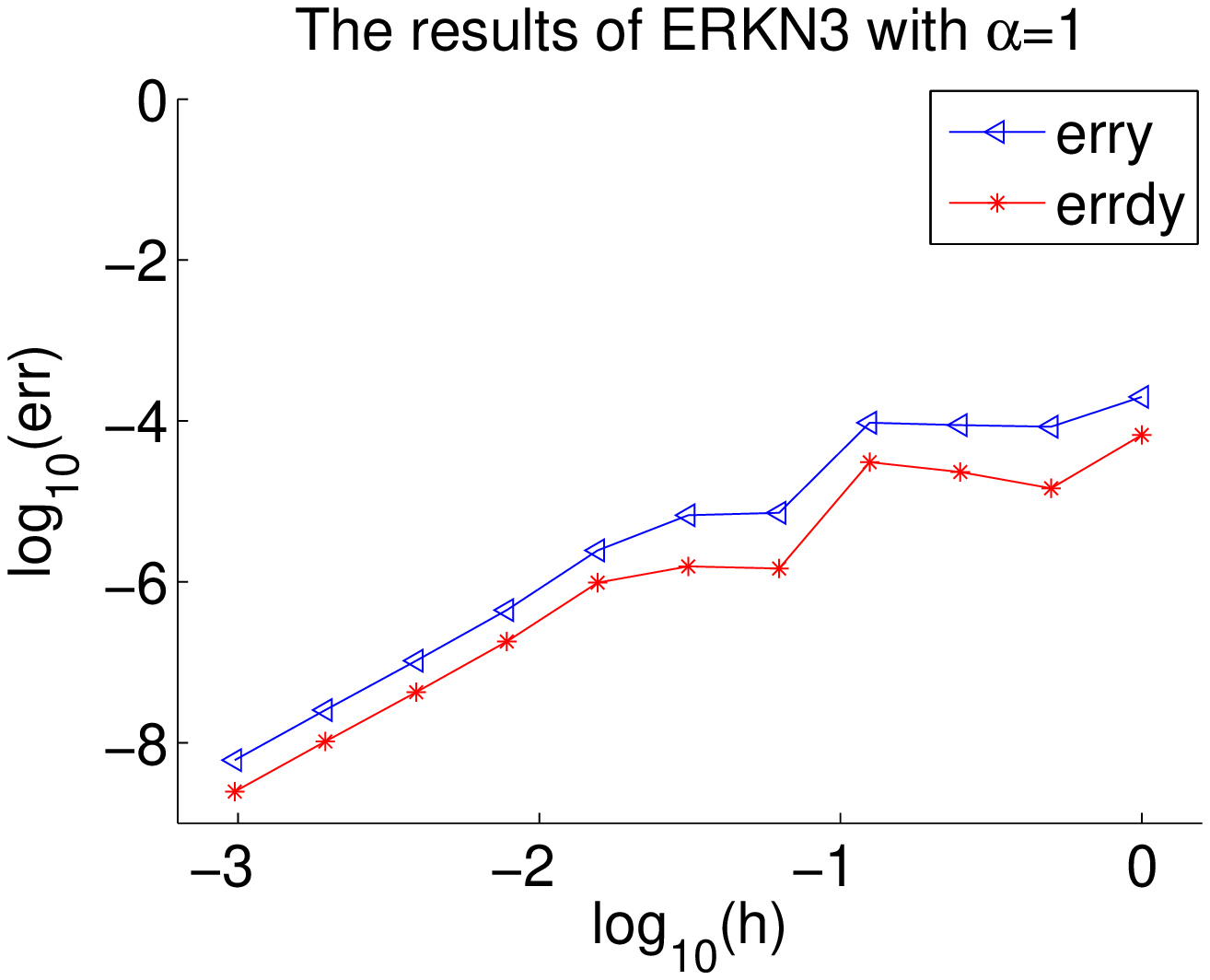}
\includegraphics[width=6cm,height=3cm]{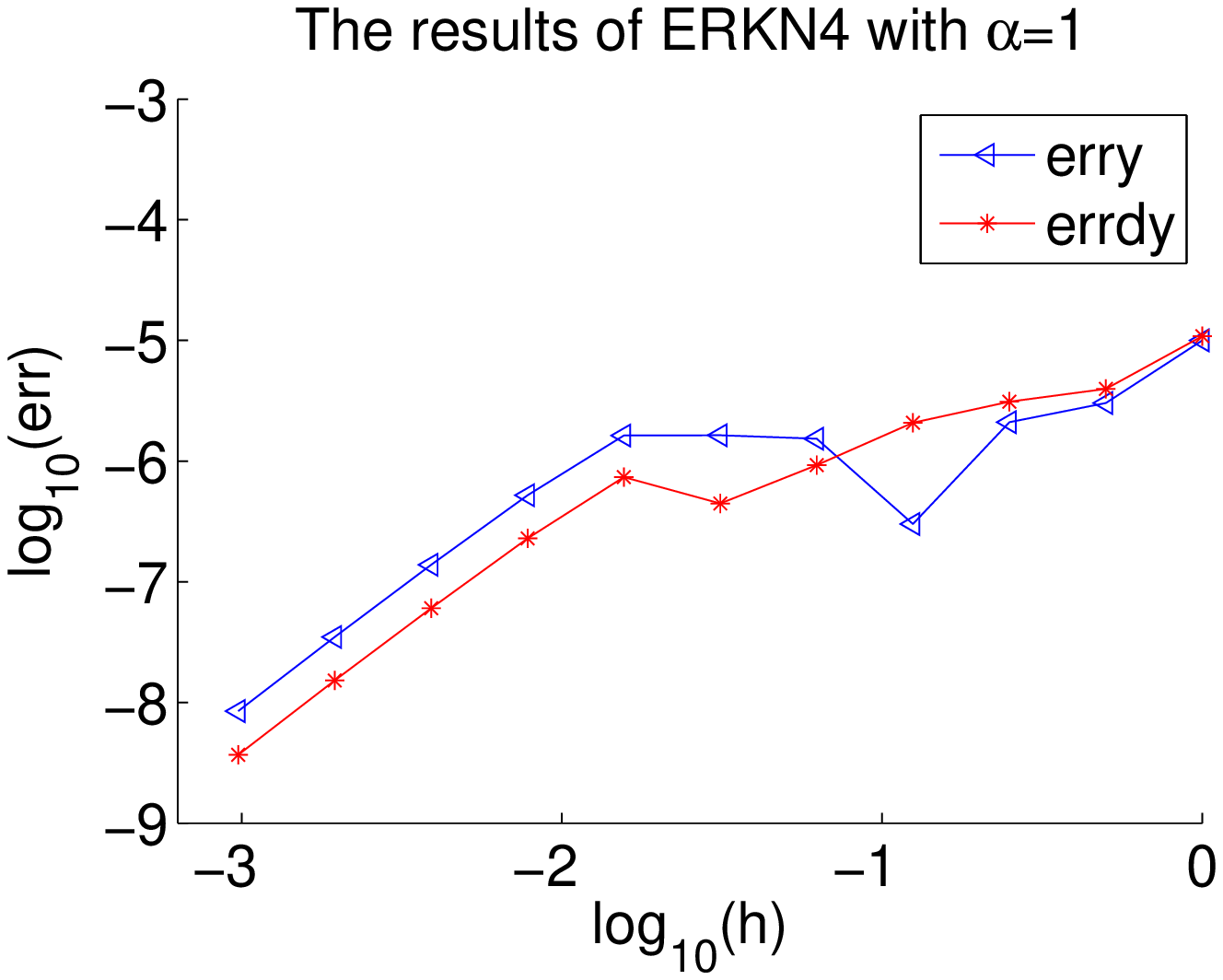}  \\
\includegraphics[width=6cm,height=3cm]{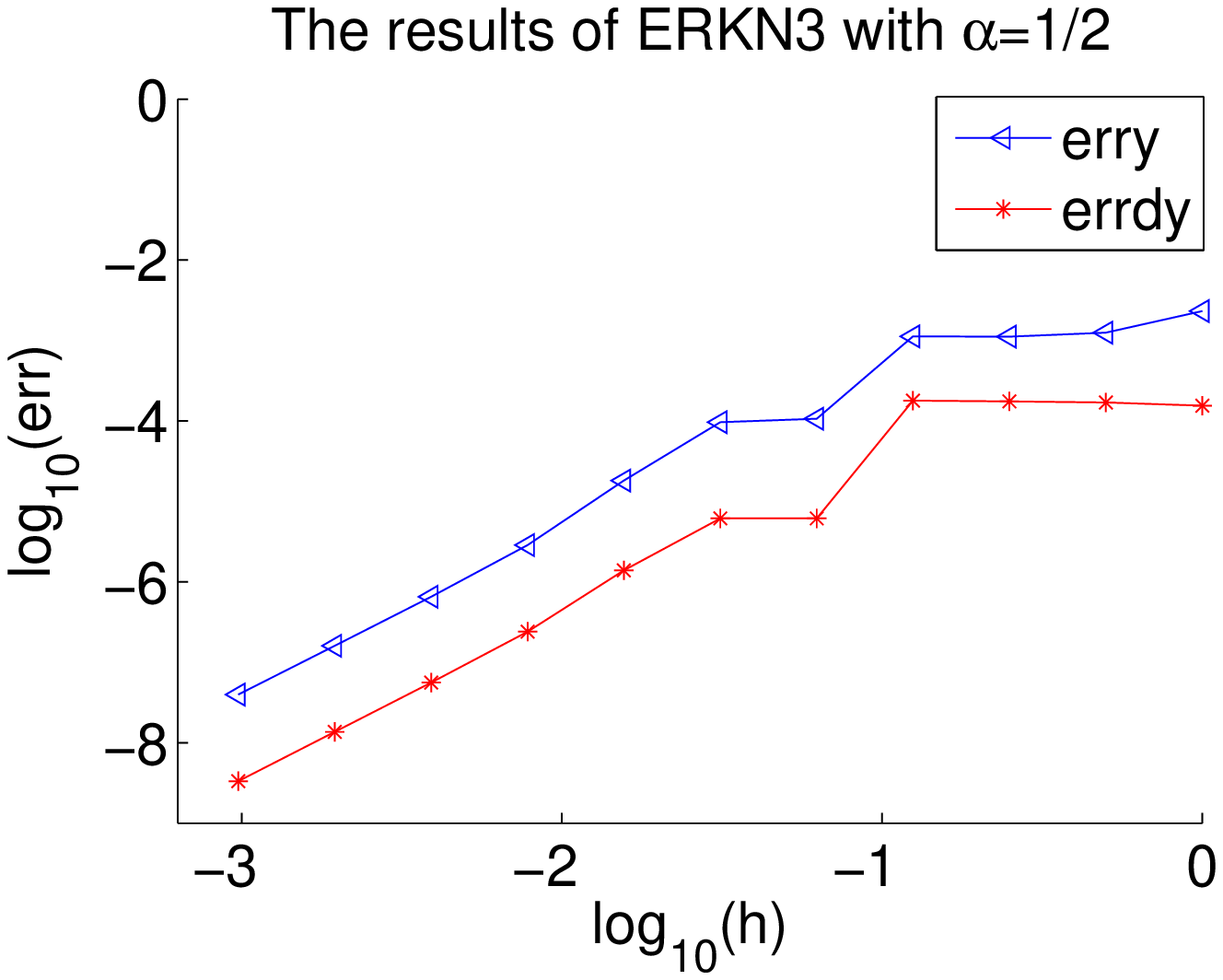}
\includegraphics[width=6cm,height=3cm]{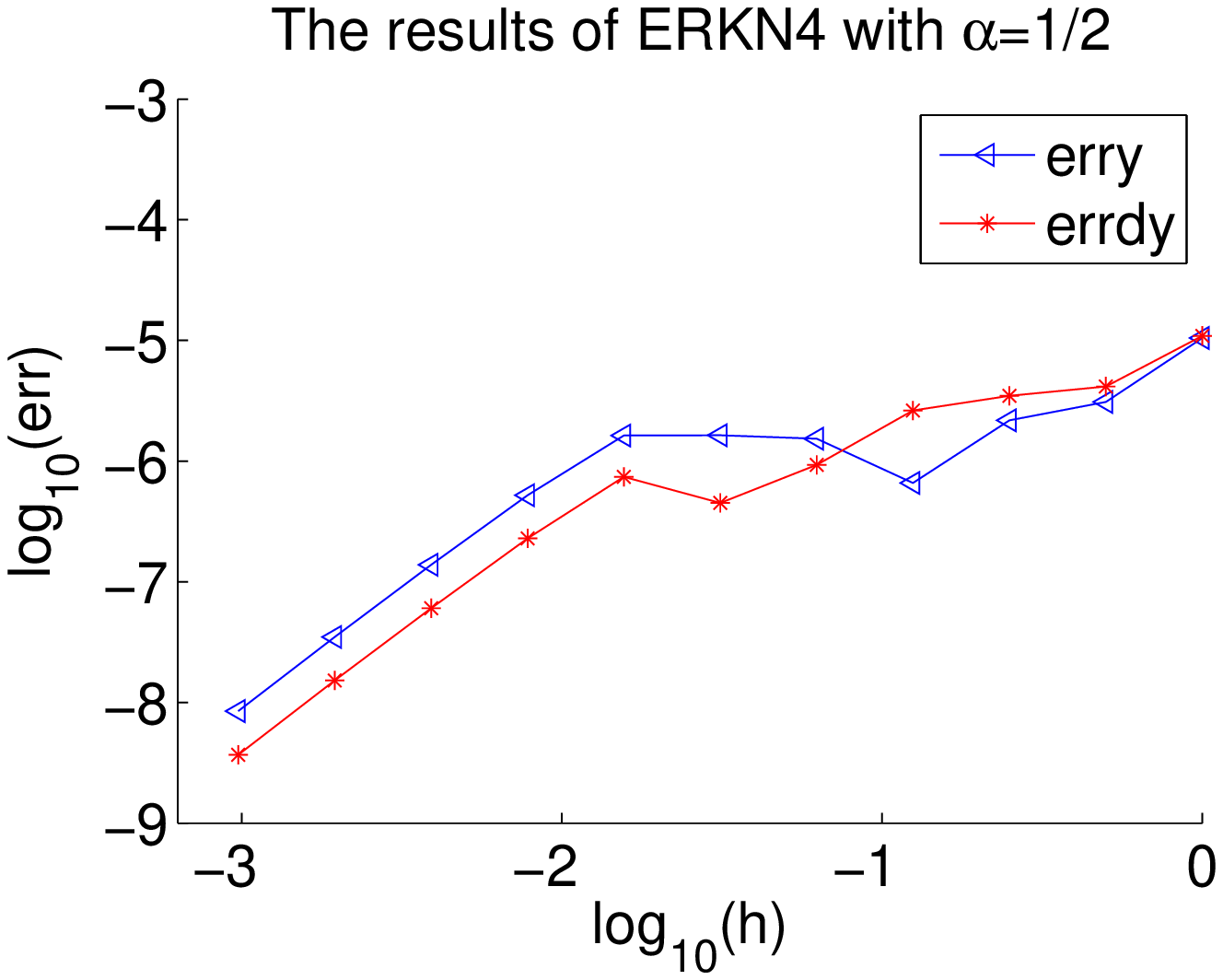}  \\
\includegraphics[width=6cm,height=3cm]{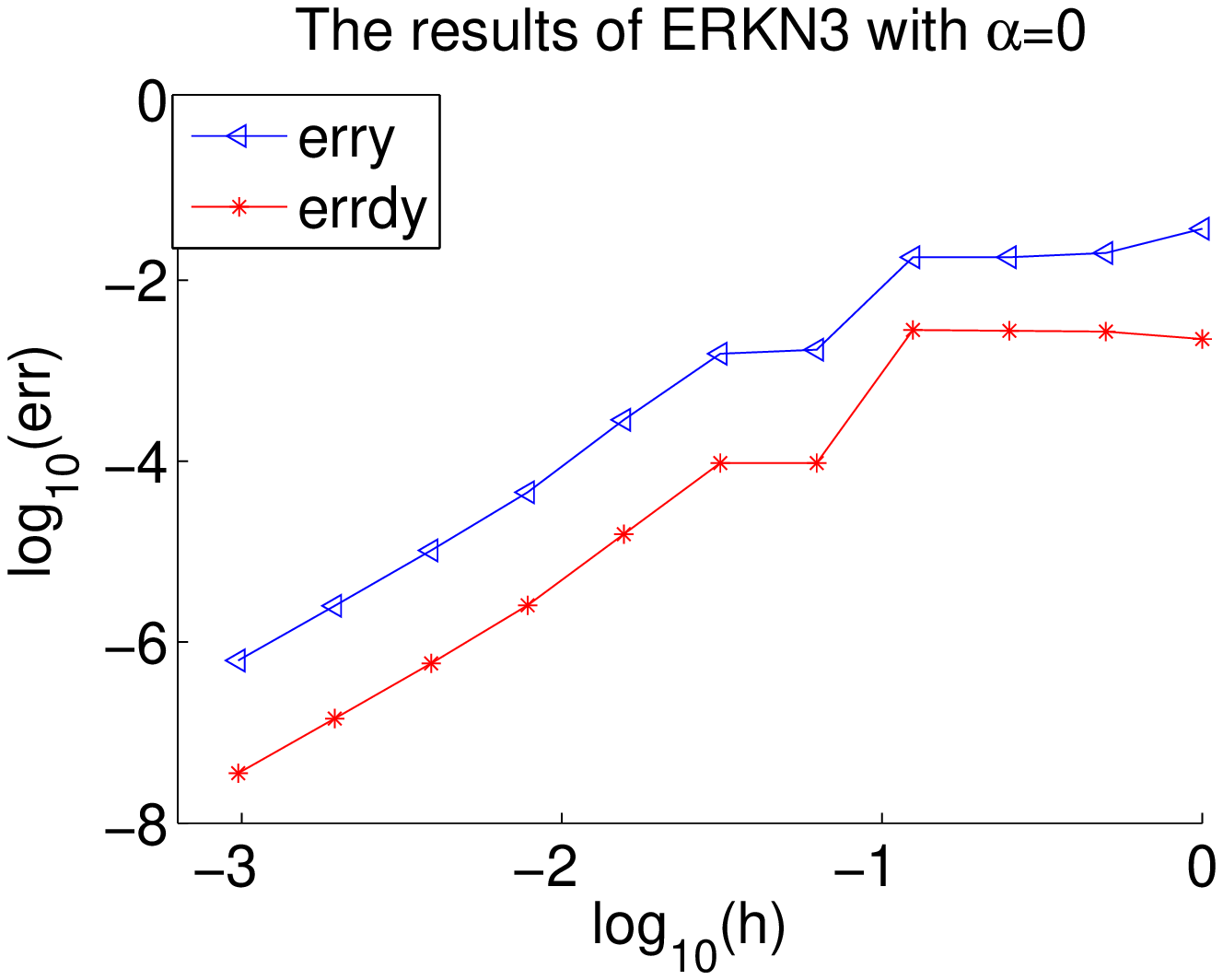}
\includegraphics[width=6cm,height=3cm]{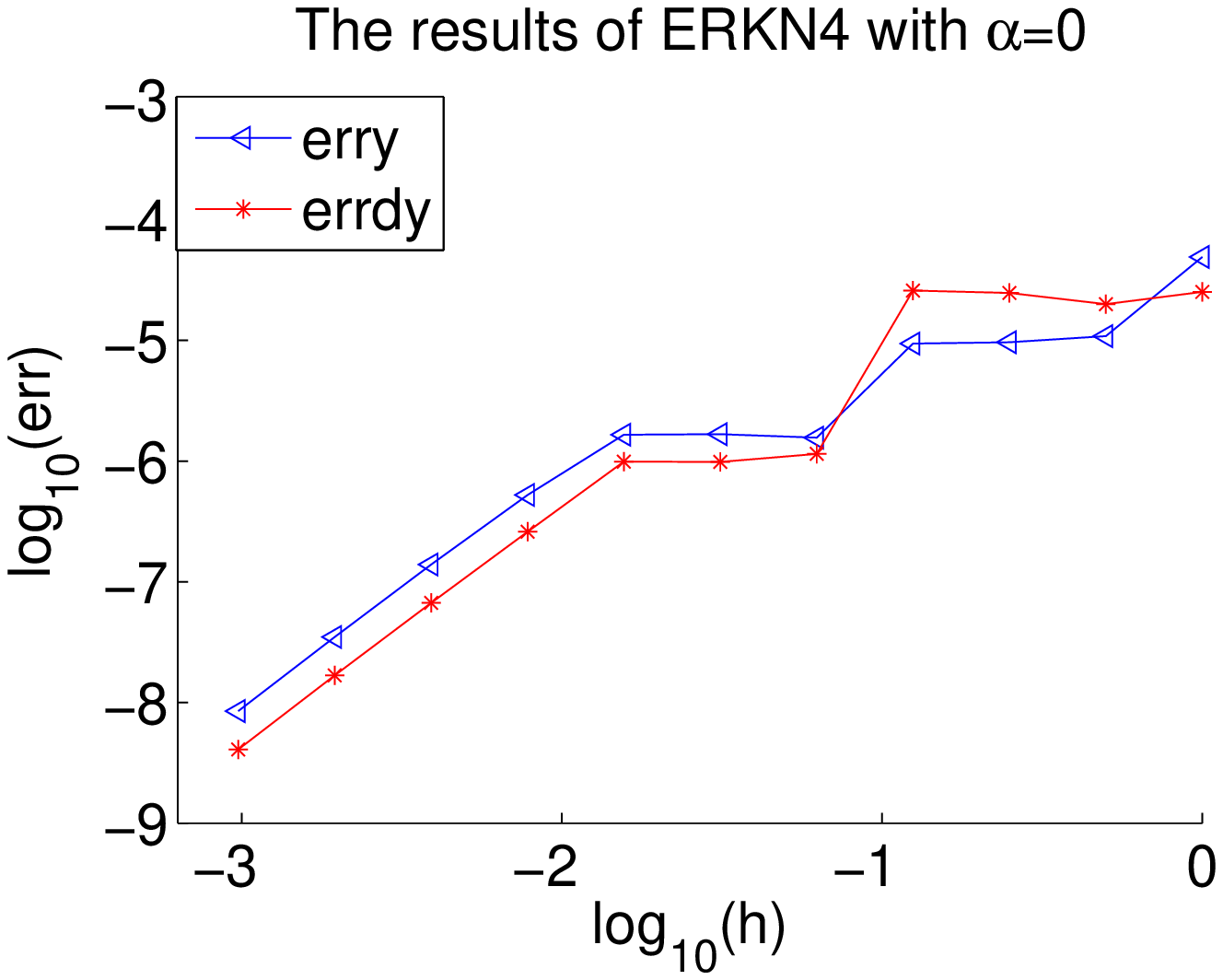}  \\
\includegraphics[width=6cm,height=3cm]{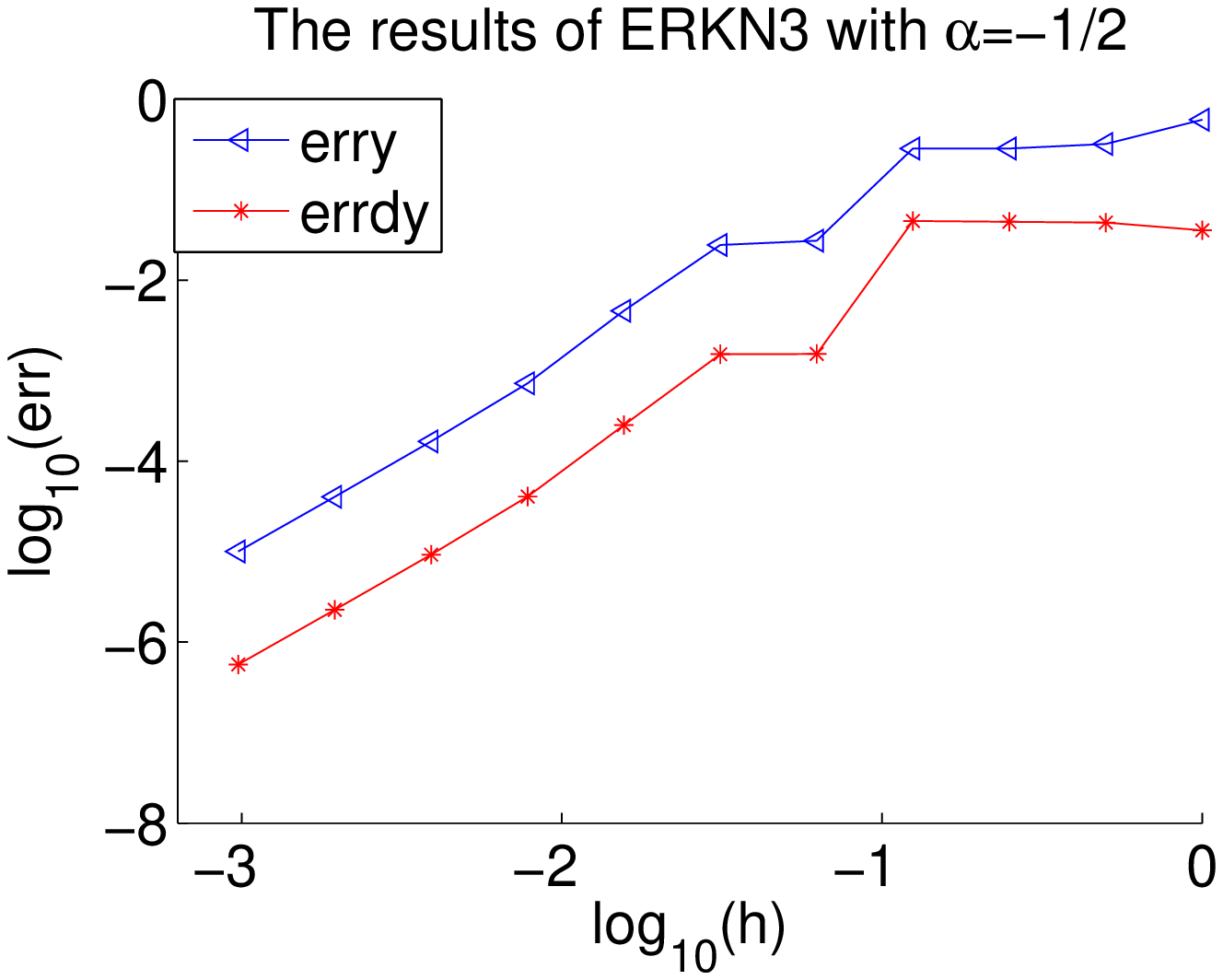}
\includegraphics[width=6cm,height=3cm]{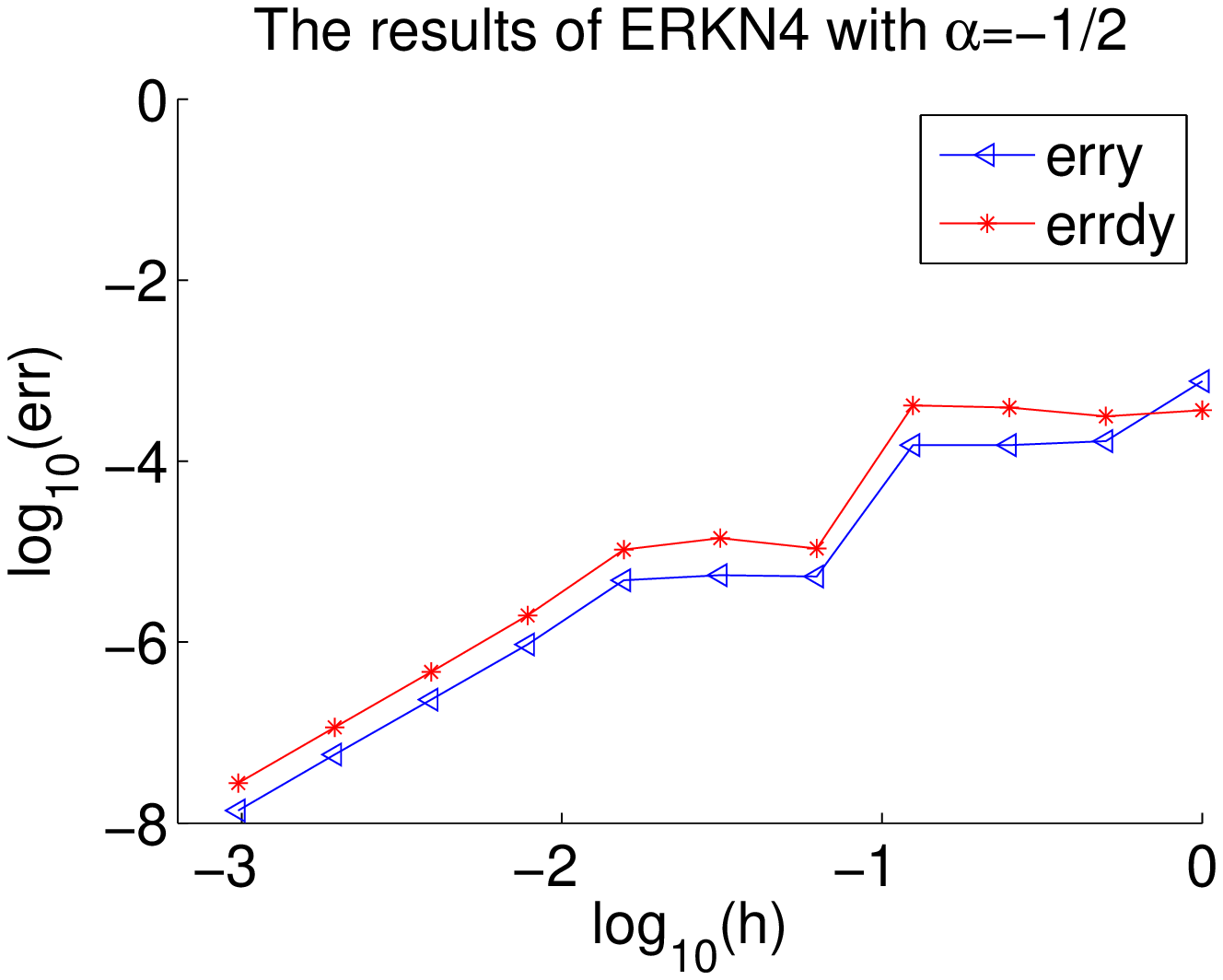}  \\
\includegraphics[width=6cm,height=3cm]{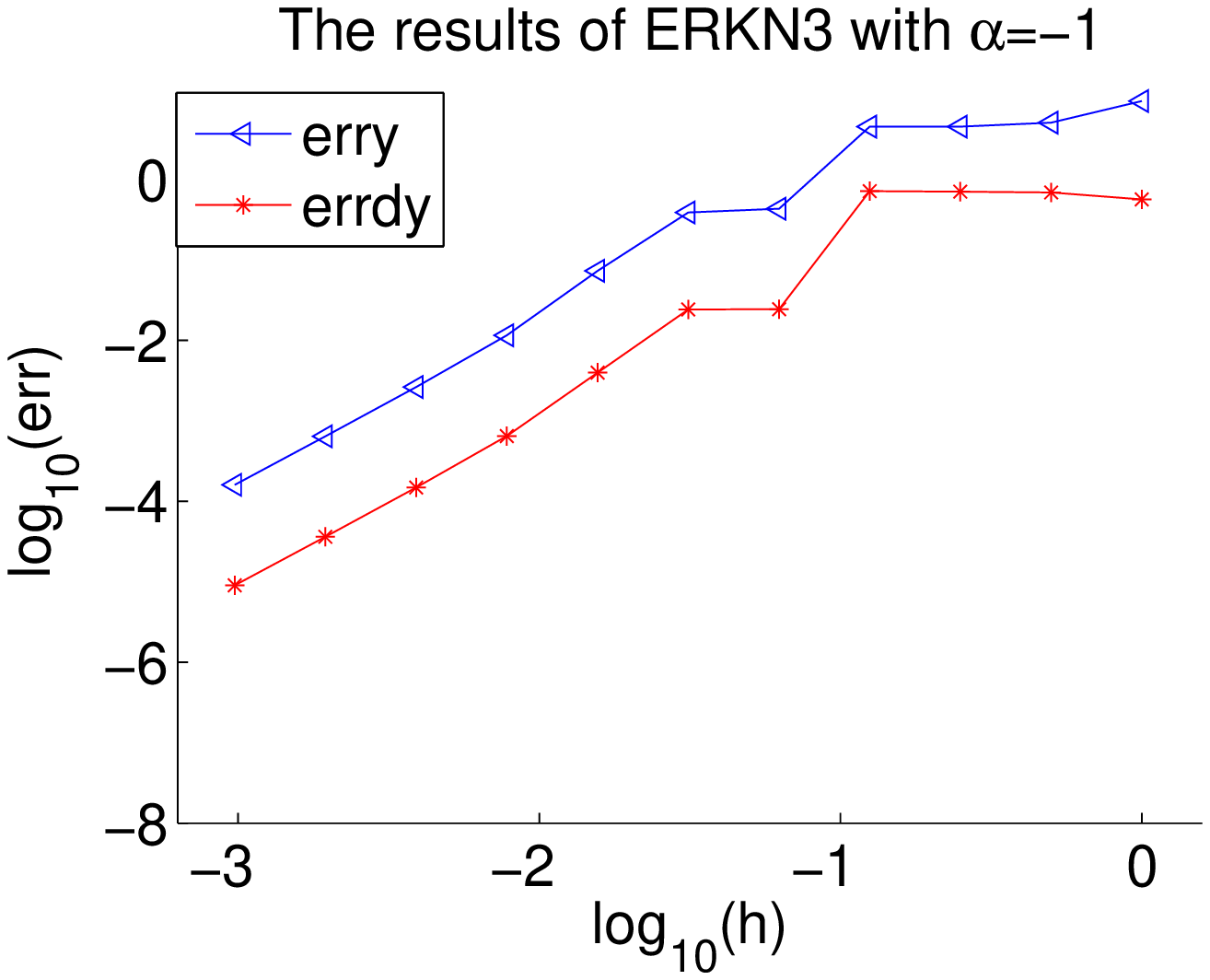}
\includegraphics[width=6cm,height=3cm]{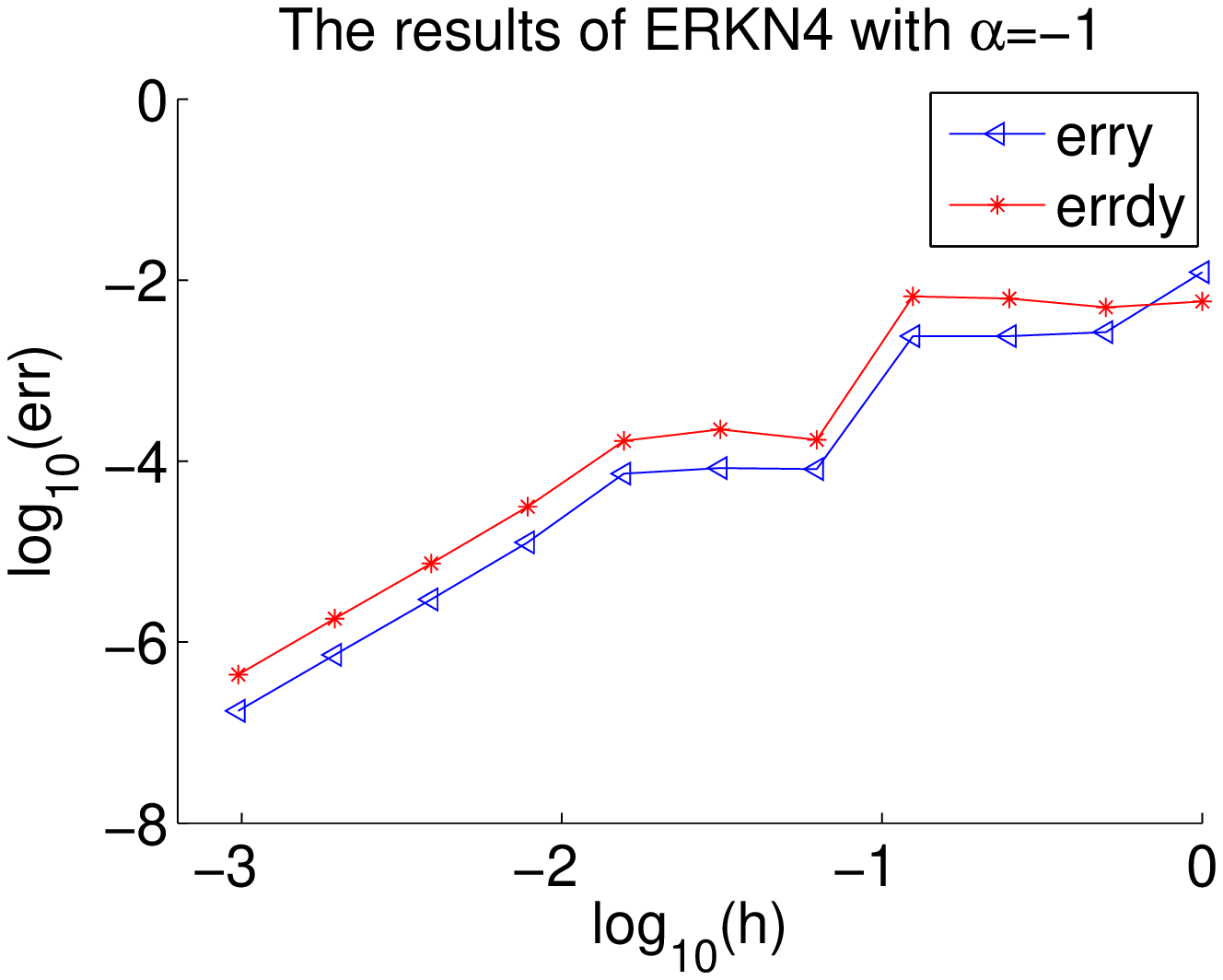}
\end{tabular}
\caption{The logarithm of the  errors    against the logarithm of stepsizes for $K=2^8$.}%
\label{fig1}%
\end{figure}

\section{Conclusions}
\label{sec:conclusions}

In this paper, we have analysed the error bounds
  of ERKN integrators when applied to spatial semidiscretizations of
semilinear wave equations. Optimal second-order convergence has been
obtained without requiring Lipschitz continuous and higher
regularity of the exact solution. Moreover, the analysis is uniform
in the spatial discretization parameter. Based on this work,  we are
hopeful of obtaining an extension to two-stage  ERKN integrators for
semidiscrete semilinear wave equations. Another issue for future
exploration is the error analysis of ERKN integrators in the case of
quasi-linear wave equations.

\end{document}